\documentclass[11pt, letterpaper, oneside, reqno]{amsart}

\usepackage[T1]{fontenc}
\usepackage[utf8]{inputenc}

\usepackage{latexsym, amsmath, amssymb, amsfonts, amscd,bm}
\usepackage{mathtools}
\usepackage{amsthm}
\usepackage{t1enc}
\usepackage[mathscr]{eucal}
\usepackage{indentfirst}
\usepackage{graphicx, pb-diagram}
\usepackage{fancyhdr}
\usepackage{fancybox}
\usepackage{enumitem}
\usepackage{xcolor}
\usepackage[all]{xy}
\usepackage{tikz-cd}
\usepackage{bbm}
\usepackage{blkarray}
\usepackage{url}
\usepackage[colorlinks = true, linkcolor = blue, urlcolor = black, citecolor = blue, anchorcolor = blue]{hyperref}
\usepackage{float}
\usepackage{bm}
\usepackage[makeroom]{cancel}
\usepackage{accents}
\usepackage{cases}

\theoremstyle{plain}
\newtheorem{thm}{Theorem}[section]

\newtheorem{prop}[thm]{Proposition}

\newtheorem*{main}{Main Theorem}
\newtheorem*{crr}{Corollary}
\newtheorem{prop/def}[thm]{Proposition/Definition}

\newtheorem{lemma}[thm]{Lemma}
\newtheorem{cor}[thm]{Corollary}

\newtheoremstyle{underline}% name
{}        % Space above, empty = `usual value'
{}              % Space below
{}              % Body font
{}    % Indent amount (empty = no indent, \parindent = para indent)
{\large}              % Thm head font
{:}             % Punctuation after thm head
{1mm}         % Space after thm head: \newline = linebreak
{{\underline{\thmname{#1}\thmnumber{ #2}}}}  % Thm head spec

\theoremstyle{underline}
\newtheorem*{claim*}{Claim}

\theoremstyle{definition}
\newtheorem{defi}[thm]{Definition}

\theoremstyle{remark}
\newtheorem{remark}[thm]{Remark}
\newtheorem{ex}[thm]{Example}

\newtheorem*{ack}{Acknowledgements}

\newcommand*{\dt}[1]{%
	\accentset{\mbox{\large\bfseries .}}{#1}}

\definecolor{forest}{rgb}{0,0.5,0}

\title{A rigidity result for coisotropic submanifolds in contact geometry}

\author{Stephane Geudens}
\address{{\scriptsize Department of Mathematics, University College London, 25 Gordon St, London WC1H0AY, UK}}
\email{\href{mailto:s.geudens@ucl.ac.uk}{\underline{\smash{s.geudens@ucl.ac.uk}}}}

\author{Alfonso G. Tortorella}
\address{{\scriptsize DipMat, Universit\`{a} degli Studi di Salerno, Via Giovanni Paolo II n.132, 84084  Fisciano, Italy}}
\email{\href{mailto:atortorella@unisa.it}{\underline{\smash{atortorella@unisa.it}}}}

\begin{document}

	\begin{abstract}
	We study coisotropic deformations of a compact regular coisotropic submanifold $C$ in a contact manifold $(M,\xi)$. Our main result states that $C$ is rigid among nearby coisotropic submanifolds whose characteristic foliation is diffeomorphic to that of $C$. When combined with a classical rigidity result for foliations, this yields conditions under which $C$ is rigid among all nearby coisotropic submanifolds.	
	\end{abstract}
	
	\maketitle
	
	\setcounter{tocdepth}{1} %doesn't display subsections in TOC 
	\tableofcontents

	\section*{Introduction}

	This paper is devoted to deformations of a compact coisotropic submanifold $C$ in a contact manifold $(M,\xi)$. We will additionally assume that the coisotropic submanifold $C$ is generic, in the sense that $C$ is transverse to the hyperplane distribution $\xi$. Such coisotropic submanifolds are called regular. The coisotropic deformation problem of a regular coisotropic submanifold $C\subset(M,\xi)$ is known to be quite involved. First, it is highly non-linear, because deformations are given by Maurer-Cartan elements of a suitable $L_{\infty}$-algebra \cite{jacobi}. Second, the deformation problem is obstructed \cite{rigidity}, meaning that there may exist first order deformations which are not tangent to any path of deformations. All of this is in sharp contrast with the deformation problem of a compact Legendrian submanifold $L\subset(M,\xi)$. The latter is linear hence unobstructed, and moreover compact Legendrians are rigid, in the sense that smooth Legendrian deformations are generated by contact isotopies.

	A regular coisotropic submanifold $C\subset(M,\xi)$ carries a characteristic foliation $\mathcal{F}$. Its properties, especially its instability, seem to play an important role in the coisotropic deformation problem of $C$, as demonstrated by the following results:
	\begin{itemize}
		\item In \cite[\S 3.1]{rigidity} the second author constructed an obstructed example, featuring a compact regular coisotropic submanifold $C$ which has arbitrarily $\mathcal{C}^{1}$-small deformations whose characteristic foliation is not diffeomorphic to that of $C$.
		\item A compact regular coisotropic submanifold $C$ is called integral if its characteristic foliation $\mathcal{F}$ is simple, i.e. given by the fibers of a fiber bundle $C\rightarrow C/\mathcal{F}$. When deformed within the class of integral coisotropic submanifolds, $C$ is rigid under contact isotopies, as proved by the second author \cite{rigidity}. In particular, this deformation problem is unobstructed. 
	\end{itemize}  
	
	These results are our motivation for considering a restricted version of the coisotropic deformation problem of $C$. Instead of allowing all $\mathcal{C}^{1}$-small coisotropic deformations of $C$, we will only consider those whose characteristic foliation is diffeomorphic to that of $C$. In other words, we deform $C$ within the class
	\[
	\text{Def}_{\mathcal{F}}(C):=\big\{C'\subset(M,\xi)\ \text{regular coisotropic}:\ \exists\ \text{foliated diffeomorphism}\ (C,\mathcal{F})\overset{\sim}{\rightarrow}(C',\mathcal{F}')\big\}.
	\] 
	
	\noindent
	Our main result (see Thm.~\ref{thm:main}) states that $C$ is rigid, when deformed inside $\text{Def}_{\mathcal{F}}(C)$.

	\begin{main}
	Let $C\subset(M,\xi)$ be a compact regular coisotropic submanifold, and let $C_t$ be a smooth path in $\text{Def}_{\mathcal{F}}(C)$ with $C_0=C$.
	Then there exists an isotopy of contactomorphisms $\psi_t\colon(M,\xi)\to(M,\xi)$, locally defined around $C$, such that $C_t=\psi_t(C)$ for small enough $t$.
	\end{main}
	
	\noindent
	Clearly, for a path of regular coisotropic submanifolds to be generated by a contact isotopy, the diffeomorphism type of the characteristic foliation needs to be constant along that path. Our result shows that constancy of the foliation is in fact the only obstruction.
	When the characteristic foliation $\mathcal{F}$ is simple, our theorem reduces to the rigidity result for integral coisotropic submanifolds from \cite{rigidity} mentioned above. 
	
	\vspace{0.2cm}
	
	Our Main Theorem is a rigidity result for the \emph{restricted} deformation problem of $C$ inside $\text{Def}_{\mathcal{F}}(C)$. It can be combined with a well-known rigidity theorem for foliations to obtain a rigidity result for the classical coisotropic deformation problem, which allows \emph{all} coisotropic deformations of $C$ (see Cor.~\ref{cor:hamilton}). The rigidity result for foliations we are referring to -- which appeared in various versions \cite{hamilton,epstein,rui} -- concerns Hausdorff foliations, i.e. those $\mathcal{F}$ for which the leaf space $C/\mathcal{F}$ is Hausdorff when endowed with the quotient topology. Such a foliation has a generic leaf $L$, to which all leaves in an open dense subset are diffeomorphic. 
	
	\begin{crr}
	Let $C\subset(M,\xi)$ be a compact, connected regular coisotropic submanifold with characteristic foliation $\mathcal{F}$. Assume that $\mathcal{F}$ is Hausdorff and that its generic leaf $L$ satisfies $H^{1}(L)=0$. If $C_t$ is a smooth path of coisotropic submanifolds with $C_0=C$, then there is an isotopy $\psi_t$ of locally defined contactomorphisms such that $C_t=\psi_t(C)$ for small $t$.
	\end{crr}
		
	\bigskip
	\textbf{Overview of the paper.}
	
	In Section~\ref{sec:One}, we collect some background information concerning coisotropic submanifolds in contact geometry. We recall that first order deformations of a regular coisotropic submanifold $C\subset(M,\xi)$, modulo those induced by contact isotopies, are given by the cohomology group $H^{1}(\mathcal{F},\ell)$. Here $\ell:=TC/(TC\cap\xi)$ is the quotient line bundle on $C$, which carries a canonical $T\mathcal{F}$-representation $\nabla^{\ell}$. With the aim of this note in mind, we also recall that first order deformations of the foliation $\mathcal{F}$, modulo those induced by isotopies, are given by the first foliated cohomology group with values in the normal bundle $H^{1}(\mathcal{F},N\mathcal{F})$.
	
	\vspace{0.2cm}
	
	In Sections~\ref{sec:two} and \ref{sec:three}, we study the deformation problem of $C$ inside $\text{Def}_{\mathcal{F}}(C)$ at the infinitesimal level.
	We first argue what are the appropriate first order deformations when deforming $C$ inside $\text{Def}_{\mathcal{F}}(C)$. By the above, these should be closed elements in $\Omega^{1}(\mathcal{F},\ell)$ which in some way give rise to a trivial class in $H^{1}(\mathcal{F},N\mathcal{F})$. To achieve this assignment, we construct a non-canonical chain map (see Def.~\ref{def:Phi} and Thm.~\ref{theor:Phi_cochain_map})
	\[
	\Phi_G\colon\big(\Omega^\bullet(\mathcal{F},\ell),d_{\nabla^\ell}\big)\longrightarrow\big(\Omega^\bullet(\mathcal{F},N\mathcal{F}),-d_{\nabla^{N\mathcal{F}}}\big),
	\]	
	which depends on a choice of complement $G$ to $T\mathcal{F}\subset TC$. We then show that first order deformations of $C$ are closed elements in $\Omega^{1}(\mathcal{F},\ell)$ whose image under $\Phi_G$ is exact (see Thm.~\ref{thm:first-order}). It turns out that the map induced in cohomology $[\Phi]:H^{1}(\mathcal{F},\ell)\rightarrow H^{1}(\mathcal{F},N\mathcal{F})$ is canonical (see Thm.~\ref{theor:Phi_cohomology}) and injective (see Prop.~\ref{prop:injective}), hence first order deformations of $C\subset\text{Def}_{\mathcal{F}}(C)$ are in fact trivial in $H^{1}(\mathcal{F},\ell)$. This means exactly that $C$ is infinitesimally rigid under contact isotopies, when deformed inside $\text{Def}_{\mathcal{F}}(C)$. In particular, this deformation problem is unobstructed (see Cor.~\ref{cor:unob}). This is the contact analog of an unobstructedness result for coisotropic submanifolds in symplectic geometry due to the first author \cite{symplectic}.
	
	\vspace{0.2cm}
	
	The infinitesimal rigidity result proved in Section~\ref{sec:three} raises the question whether $C$ is actually rigid inside $\text{Def}_{\mathcal{F}}(C)$. We show in Section~\ref{sec:rigidity} that this is indeed the case, by establishing our Main Theorem mentioned above. The proof relies on an extension of Gray stability for contact structures to the case of pre-contact structures (see Prop.~\ref{prop:isotopy}) and a uniqueness result concerning coisotropic embeddings in contact manifolds (see Prop.~\ref{prop:coisotropic_embedding}). In Section~\ref{sec:five}, we prove the above Corollary concerning coisotropic submanfolds with Hausdorff foliation.
	
	\vspace{0.2cm}
	
	Throughout the paper, we make frequent use of the fact that (pre-)contact structures can be viewed as (pre-)symplectic forms on a certain Lie algebroid, called Atiyah algebroid \cite{luca}. This point of view allows one to study (pre-)contact structures using (pre-)symplectic geometry. We recall the details of this correspondence in the Appendix.

	\begin{ack}
	S.G. acknowledges support from the UCL Institute for Mathematics \& Statistical Science (IMSS). A.G.T.~is a member of the National Group for Algebraic and Geometric Structures, and their Applications (GNSAGA - INdAM).
	\end{ack}

	\section{Background and setup}\label{sec:One}

	\subsection{Coisotropic submanifolds in contact geometry}\label{sec:one}
	\vspace{0.1cm}
	\noindent
	
	We recall some background information about contact manifolds and their coisotropic submanifolds, following \cite[\S 5]{jacobi}. Let $M$ be a manifold with a hyperplane distribution $\xi$. Denoting the quotient line bundle by $L:=TM/\xi$, the distribution $\xi$ is the kernel of the \textbf{structure form} $\theta\in\Gamma(T^{*}M\otimes L)$ of $(M,\xi)$, given by
	\[
	\theta:TM\rightarrow L: X\mapsto X\ \text{mod}\ \xi.
	\] 
	The \textbf{curvature form} $\omega\in\Gamma(\wedge^{2}\xi^{*}\otimes L)$ of $(M,\xi)$ is defined by
	\[
	\omega(X,Y)=\theta([X,Y]),\hspace{1cm}X,Y\in\Gamma(\xi),
	\]
	and we denote the associated vector bundle morphism by
	\[
	\omega^{\flat}:\xi\rightarrow \xi^{*}\otimes L:X\mapsto \omega(X,-).
	\]
	The characteristic distribution of $(M,\xi)$ is given by $\ker\omega^{\flat}$. It is an involutive distribution which is singular in general.
	
	\begin{defi}
		We call $(M,\xi)$ a \textbf{pre-contact} manifold if the characteristic distribution $\ker\omega^{\flat}$ has constant rank. In that case, we refer to $\text{rk}(\ker\omega^{\flat})$ as the \textbf{rank} of $(M,\xi)$. The foliation $\mathcal{F}$ integrating $\ker\omega^{\flat}$ is the \textbf{characteristic foliation} of $(M,\xi)$.
	\end{defi}
	
	The curvature form $\omega$ of $(M,\xi)$ vanishes exactly when $\xi$ is integrable. Consequently, we say that $\xi$ is maximally non-integrable if $\omega$ is non-degenerate, i.e. $\ker\omega^{\flat}$ is trivial.
	
	\begin{defi}
		We call $(M,\xi)$ a \textbf{contact} manifold if $\xi$ is maximally non-integrable.
	\end{defi}
	
	%\begin{remark}\label{rem:structure-curvature}
	%Choose a local defining one-form $\alpha\in\Omega^{1}(M)$ for $\xi$ and trivialize the line bundle $L$ accordingly via
	%\[
	%L\rightarrow M\times\mathbb{R}:X\ \text{mod}\ \xi\mapsto\alpha(X).
	%\]
	%The structure form $\theta\in\Gamma(T^{*}M\otimes L)$ is then just $\alpha$, and the curvature form  $\omega\in\Gamma(\wedge^{2}\xi^{*}\otimes L)$ is the restriction of $-d\alpha$ to $\xi$. The fact that $\omega$ is maximally non-integrable amounts to $\alpha\wedge (d\alpha)^{n}$ being nowhere zero, where $\dim(M)=2n+1$.
	%\end{remark}
	
	Let $C$ be a submanifold of a contact manifold $(M,\xi)$. The intersection $\xi_C:=TC\cap\xi$ is a generically singular distribution on $C$ -- if $C$ is transverse to $\xi$, then $\xi_C$ is a hyperplane distribution on $C$. Denote by $\xi_{C}^{\perp_{\omega}}$ the bundle (possibly of varying rank) of $\omega$-orthogonals of $\xi_C$ in $(\xi,\omega)$, i.e.
	\[
	\xi_{C}^{\perp_{\omega}}=\{v\in\xi:\ \omega(v,w)=0\ \ \forall w\in\xi_C\}.
	\] 
	
	%By symplectic linear algebra, the following are equivalent:
	%\begin{enumerate}[label=\roman*)]
	%\item $(\xi_C)_{p}$ is a coisotropic subspace of $(\xi_p,\omega_p)$ for all $p\in C$, i.e. $\xi_{C}^{\perp_{\omega}}\subset\xi_C$.
	%\item $\xi_C$ is a pre-contact structure on $C$, with characteristic distribution $\xi_{C}^{\perp_{\omega}}$.
	%\end{enumerate}

	\begin{defi}
		A submanifold $C$ of a contact manifold $(M,\xi)$ is \textbf{coisotropic} if $(\xi_C)_{p}$ is a coisotropic subspace of $(\xi_p,\omega_p)$ for all $p\in C$, i.e. $\xi_{C}^{\perp_{\omega}}\subset\xi_C$. We call $C$ \textbf{regular coisotropic} if $C$ is coisotropic and transverse to $\xi$.
	\end{defi}
	
	In this note, we focus on regular coisotropic submanifolds.
	
	\begin{remark}\label{rem:precontact}
		A submanifold $C^{k+1}\subset(M^{2n+1},\xi)$ is regular coisotropic exactly when $\xi_C$ is a pre-contact structure of rank $2n-k$. Indeed, because $\ker(\omega|_{\xi_C}^{\flat})=\xi_C\cap\xi_{C}^{\perp_{\omega}}$, we have the following equivalences:
		\begin{align*}
			&C^{k+1}\subset(M^{2n+1},\xi)\ \text{is regular coisotropic}\\ &\hspace{3cm}\Leftrightarrow \dim(\xi_C)_{p}=k\ \text{and}\   \big(\xi_C\cap\xi_{C}^{\perp_{\omega}}\big)_{p}=\big(\xi_{C}^{\perp_{\omega}}\big)_{p}\ \ \ \forall p\in C,\\
			&\hspace{3cm}\Leftrightarrow \dim(\xi_C)_{p}=k\ \text{and}\  \dim\big(\xi_C\cap\xi_{C}^{\perp_{\omega}}\big)_{p}=2n-k \ \ \ \forall p\in C,\\
			&\hspace{3cm}\Leftrightarrow \xi_C\ \text{is pre-contact of rank}\ 2n-k.
		\end{align*}
		Hence, a regular coisotropic submanifold $C^{k+1}\subset(M^{2n+1},\xi)$ carries a characteristic foliation $\mathcal{F}$, which integrates the distribution $\xi_{C}^{\perp_{\omega}}$ and has $(2n-k)$-dimensional leaves.
	\end{remark}

	We also mention an alternative characterization of regular coisotropic submanifolds, which will be important in Section \ref{sec:rigidity}. It is reminiscent of the definition of coisotropic submanifolds in symplectic geometry, relying on the fact that contact structures can be viewed as symplectic Atiyah forms. For the necessary background and notation concerning Atiyah forms, we refer to the Appendix. 
	
	\begin{prop}[\cite{rigidity}]\label{prop:coiso-atiyah}
		Let $(M,\xi)$ be a contact manifold with quotient line bundle $L$, and denote by $\Omega\in\Omega^{2}_{D}(L)$ its symplectic Atiyah form. Let $C\subset (M,\xi)$ be a submanifold transverse to $\xi$, with restricted line bundle $L_{C}:=L|_{C}$. Then the following are equivalent:
		\begin{enumerate}
			\item $C$ is a regular coisotropic submanifold of $(M,\xi)$,
			\item the pullback $i^{*}\Omega\in\Omega^{2}_{D}(L_{C})$ is a pre-symplectic Atiyah form with kernel $(DL_{C})^{\perp_{\Omega}}$. Here $(DL_{C})^{\perp_{\Omega}}$ denotes the $\Omega$-orthogonal of $DL_{C}$ inside $DL$.
		\end{enumerate}
	\end{prop}

	A special case of \cite[Thm.~3]{contact-thickening} provides a normal form for a contact manifold $(M,\xi)$ around a regular coisotropic submanifold $C\subset(M,\xi)$, similar to Gotay's theorem in symplectic geometry \cite{gotay}. Before introducing the local model, we have to fix some notation. Let $\mathcal{F}$ be the characteristic foliation of the pre-contact manifold $(C,\xi_C)$. Denote by $\ell:=TC/\xi_C$ its line bundle, and by $\theta_C\in\Omega^{1}(C,\ell)$ its structure form. The local model lives on a neighborhood of the zero section of the vector bundle $p:T^{*}\mathcal{F}\otimes \ell\rightarrow C$. Its construction depends on a choice of complement $G$ to the characteristic distribution $T\mathcal{F}$, i.e. $TC=T\mathcal{F}\oplus G$. The model involves the following elements of $\Omega^{1}\big(T^{*}\mathcal{F}\otimes \ell,p^{*}\ell\big)$:
	\begin{enumerate}[label=\roman*)]
		\item We pull back the structure form $\theta_C\in\Omega^{1}(C,\ell)$ to $p^{*}\theta_C\in\Omega^{1}\big(T^{*}\mathcal{F}\otimes \ell,p^{*}\ell\big)$,
		\item We define $\theta_G\in\Omega^{1}\big(T^{*}\mathcal{F}\otimes \ell,p^{*}\ell\big)$ as follows. For $\alpha\in (T^{*}\mathcal{F}\otimes \ell)_{x}$ and $v\in T_{\alpha}(T^{*}\mathcal{F}\otimes \ell)$,
		\[
		(\theta_G)_{\alpha}(v)=\big\langle \alpha, \text{pr}_{T\mathcal{F}}(dp(v))\big\rangle.
		\] 
		Here $\text{pr}_{T\mathcal{F}}:TC\rightarrow T\mathcal{F}$ is the projection in the splitting $TC=T\mathcal{F}\oplus G$.
	\end{enumerate}
	\noindent
	It turns out that the kernel of
	\[
	\gamma_G:=p^{*}\theta_C+\theta_G\in\Omega^{1}\big(T^{*}\mathcal{F}\otimes \ell,p^{*}\ell\big)
	\]
	defines a contact structure on a neighborhood $U$ of $C\subset T^{*}\mathcal{F}\otimes \ell$. We call $(U,\ker\gamma_G)$ the \textbf{contact thickening} of $(C,\xi_C)$. This is the local model for $(M,\xi)$ around $C$.
	
	\begin{thm}\cite{contact-thickening}\label{thm:normalform}
		If $C\subset(M,\xi)$ is a regular coisotropic submanifold, then a neighborhood of $C$ in $M$ can be identified with the contact thickening $(U,\ker\gamma_G)$.
	\end{thm}
	
	%\begin{remark}\label{rem:model}
	%For computational purposes, it is useful to note what $(U,\ker\gamma_G)$ looks like if we locally trivialize the line bundle $l$ using a defining one-form $\alpha_C\in\Omega^{1}(C)$ for $\xi_C$. The contact thickening $(U,\ker\gamma_G)$ is then equivalent with a contact structure on a neighborhood $V$ of the zero section of $p:T^{*}\mathcal{F}\rightarrow C$, given by
	%\begin{equation}\label{eq:gotay-like}
	%\big(U,\ker\gamma_G\big)\cong\big(V,\ker(p^{*}\alpha_C+\epsilon_{G}^{*}\theta_{can})\big).
	%\end{equation}
	%Here $\epsilon_{G}:T^{*}\mathcal{F}\hookrightarrow T^{*}C$ is the inclusion induced by the complement $G$, and $\theta_{can}$ denotes the tautological one-form on $T^{*}C$. The contact structure \eqref{eq:gotay-like} is clearly a direct analog of Gotay's local model around coisotropic submanifolds in symplectic geometry \cite{gotay}.
	%\end{remark}
	
	By Thm.~\ref{thm:normalform}, studying $\mathcal{C}^{1}$-small coisotropic deformations of a regular coisotropic submanifold $C\subset(M,\xi)$ amounts to studying sections of the contact thickening $(U,\ker\gamma_G)$ whose graph is coisotropic. Note here that small enough coisotropic deformations of $C$ are automatically regular. The following result describes regular coisotropic sections of $(U,\ker\gamma_G)$. 
	\begin{lemma}\label{lem:coiso-in-model}
		Let $C^{k+1}\subset(M^{2n+1},\xi)$ be a regular coisotropic submanifold with local model $(U,\ker\gamma_G)$. For any section $\eta\in\Gamma(U)$, the following are equivalent:
		\begin{enumerate}[label=\roman*)]
			\item $\text{Graph}(\eta)\subset(U,\ker\gamma_G)$ is regular coisotropic,
			\item $\ker\big(\theta_C+\operatorname{pr}_{T\mathcal{F}}^{*}\eta\big)$ is a pre-contact structure of rank $(2n-k)$ on $C$.
		\end{enumerate}
	\end{lemma}
	
	Recall that we fixed a complement $TC=T\mathcal{F}\oplus G$. The map $\operatorname{pr}_{T\mathcal{F}}^{*}$ appearing in $ii)$ above is the dual of the projection  $\operatorname{pr}_{T\mathcal{F}}:TC\rightarrow T\mathcal{F}$ along $G$. %$\epsilon_{G}:T^{*}\mathcal{F}\hookrightarrow T^{*}C$ is the inclusion coming from the choice of complement $G$. It induces a map $\epsilon_{G}\otimes\text{Id}:\Gamma(T^{*}\mathcal{F}\otimes l)\hookrightarrow \Gamma(T^{*}C\otimes l)$ which appears in $ii)$ above.
	
	\begin{proof}
		The section $\eta\in\Gamma(U)$ defines a diffeomorphism
		\[
		\sigma_{\eta}:T^{*}\mathcal{F}\otimes \ell\rightarrow T^{*}\mathcal{F}\otimes \ell:(x,\alpha)\mapsto (x,\alpha+\eta(x)),
		\]	
		which is just the translation by $\eta$. Note that $\sigma_\eta$ preserves the fibers of $p:T^{*}\mathcal{F}\otimes \ell\rightarrow C$, hence it induces a well-defined pullback on $\Omega^{1}(T^{*}\mathcal{F}\otimes \ell,p^{*}\ell)$. Since $\sigma_\eta$ takes $C$ to $\text{graph}(\eta)$, we have that
		$\text{graph}(\eta)$ is regular coisotropic for $\ker\gamma_G$ exactly when $C$ is regular coisotropic for $(\sigma_{-\eta})_{*}\big(\ker\gamma_G\big)=\ker\big(\sigma_{\eta}^{*}\gamma_G\big)$. We claim to have the following expression for $\sigma_{\eta}^{*}\gamma_G$.
		
		\vspace{0.2cm}
		\noindent
		\underline{Claim:} $\sigma_{\eta}^{*}\gamma_G=\gamma_G+p^{*}\big(\operatorname{pr}_{T\mathcal{F}}^{*}\eta\big)$.
		
		\vspace{0.1cm}
		\noindent	
		To prove the claim, we only have to check that $\sigma_{\eta}^{*}\theta_G=\theta_G+p^{*}\big(\operatorname{pr}_{T\mathcal{F}}^{*}\eta\big)$. To this end, we compute for $\alpha\in(T^{*}\mathcal{F}\otimes \ell)_{x}$ and $v\in T_{\alpha}(T^{*}\mathcal{F}\otimes \ell)$,
		\begin{align*}
			\big(\sigma_{\eta}^{*}\theta_G\big)_{\alpha}(v)&=\big(\theta_G\big)_{\alpha+\eta(x)}((\sigma_{\eta})_{*}v)\\
			&=\big\langle \alpha+\eta(x), \text{pr}_{T\mathcal{F}}(p_{*}(\sigma_{\eta})_{*}v)\big\rangle\\
			&=\big(\theta_G\big)_{\alpha}(v)+\big\langle \eta(x),\text{pr}_{T\mathcal{F}}(p_{*}v)\big\rangle\\
			&=\big(\theta_G\big)_{\alpha}(v)+\big(p^{*}\big(\operatorname{pr}_{T\mathcal{F}}^{*}\eta\big)\big)_{\alpha}(v).
		\end{align*} 
		The result now follows from Rem.~\ref{rem:precontact}, since the pullback of $\sigma_{\eta}^{*}\gamma_G$ to $C$ reads
		\[
		i^{*}(\sigma_{\eta}^{*}\gamma_G)=\theta_C+\operatorname{pr}_{T\mathcal{F}}^{*}\eta.\qedhere
		\]
	\end{proof}
	
	\begin{remark}\label{rem:transverse}
		The proof of Lemma \ref{lem:coiso-in-model} shows that the graph of any section $\eta\in\Gamma(U)$ is automatically transverse to the hyperplane distribution $\ker\gamma_G$. Therefore, the regularity requirement is automatically met. To see why $\text{graph}(\eta)\pitchfork \ker\gamma_G$, note that this is equivalent with $C\pitchfork\ker\big(\gamma_G+p^{*}\big(\operatorname{pr}_{T\mathcal{F}}^{*}\eta\big)\big)$, which in turn is equivalent with $\theta_C+\operatorname{pr}_{T\mathcal{F}}^{*}\eta$ being nowhere zero. The latter statement holds, for if $\theta_C+\operatorname{pr}_{T\mathcal{F}}^{*}\eta$ would vanish at some point, then $\theta_C$ would vanish there, which is impossible. 
	\end{remark}
	
	Lemma \ref{lem:coiso-in-model} has the advantage that it describes regular coisotropic deformations of $C$ in terms of information that lives on the manifold $C$ itself. It
	shows in particular that small coisotropic deformations of $C$ induce deformations of the pre-contact structure $\xi_C$ with the same rank, i.e. there is an assignment
	\[
	\text{Def}_{coiso}(C)\rightarrow\text{Def}_{pre-contact}(\xi_C):\eta\mapsto \ker\big(\theta_C+\operatorname{pr}_{T\mathcal{F}}^{*}\eta\big).
	\]

	\subsection{Coisotropic deformations with fixed characteristic foliation}
	
	\vspace{0.1cm}
	\noindent
	
	Assume we are given a regular coisotropic submanifold $C\subset(M,\xi)$ with characteristic foliation $\mathcal{F}$. We aim to study coisotropic deformations $C'$ of $C$, which are small enough so that they remain regular and whose characteristic foliation $\mathcal{F}'$ is diffeomorphic to $\mathcal{F}$. More precisely, we are interested in $\mathcal{C}^{1}$-small elements of the following space.
	
	\begin{defi}\label{def:Def}
		We define $\text{Def}_{\mathcal{F}}(C)$ to be the space of regular coisotropic submanifolds $C'\subset(M,\xi)$ for which there exists a foliated diffeomorphism $(C,\mathcal{F})\overset{\sim}{\rightarrow}(C',\mathcal{F}')$.
	\end{defi}
	
	Thanks to Thm.~\ref{thm:normalform}, we can pass from $(M,\xi)$ to the contact thickening $(U,\gamma_G)$. Doing so turns the problem into studying coisotropic sections $\eta\in\Gamma(U,\gamma_G)$ whose characteristic foliation is diffeomorphic to $\mathcal{F}$. In this respect, we remark the following.
	
	\begin{lemma}\label{lem:diffeos}
		For any coisotropic section $\eta\in\Gamma(U,\gamma_G)$, the following are equivalent:
		\begin{enumerate}
			\item there is a diffeomorphism $\psi:C\rightarrow\text{graph}(\eta)$ which takes $\mathcal{F}$ to the characteristic foliation of the pre-contact structure $\ker(i^{*}\gamma_G)$,
			\item there is a diffeomorphism $\phi:C\rightarrow C$ which takes $\mathcal{F}$ to to the characteristic foliation of the pre-contact structure $\ker\big(\theta_C+\operatorname{pr}_{T\mathcal{F}}^{*}\eta\big)$.
		\end{enumerate}
	\end{lemma}
	\begin{proof}
		The section $\eta$ defines a diffeomorphism $\tau_{\eta}:C\rightarrow\text{graph}(\eta)$ which, as shown in the proof of Lemma \ref{lem:coiso-in-model}, pulls back $i^{*}\gamma_G\in\Omega^{1}(\text{graph}(\eta),p^{*}\ell|_{\text{graph}(\eta)})$ to $\theta_C+\operatorname{pr}_{T\mathcal{F}}^{*}\eta\in\Omega^{1}(C,\ell)$. The inverse of $\tau_{\eta}$ is just the restriction of the bundle projection $p|_{\text{graph}(\eta)}$. Hence, given $\psi:C\rightarrow\text{graph}(\eta)$ as in $(1)$ above, we can set $\phi:=p|_{\text{graph}(\eta)}\circ\psi$. Conversely, given given $\phi$ as in $(2)$ above, we can set $\psi=\tau_{\eta}\circ\phi$.
	\end{proof}
	
	With Lemma \ref{lem:diffeos} in mind, we now introduce a space $\text{Def}^{U}_{\mathcal{F}}(C)$ which serves as a local model for $\text{Def}_{\mathcal{F}}(C)$. The space $\text{Def}^{U}_{\mathcal{F}}(C)$ is the main object of study in this note.
	
	\begin{defi}\label{def:DefU}
		Let $C\subset(M,\xi)$ be a regular coisotropic submanifold with characteristic foliation $\mathcal{F}$ and contact thickening $(U,\gamma_G)$. We define $\text{Def}^{U}_{\mathcal{F}}(C)$ to be the space of sections $\eta\in\Gamma(U)$ such that
		\begin{equation*}
			\begin{cases}
				\text{Graph}(\eta)\ \text{is coisotropic in}\ (U,\gamma_{G})\\
				\text{There is a foliated diffeomorphism}\  (C,\mathcal{F})\overset{\sim}{\rightarrow}\big(C,\mathcal{F}_{\eta}\big)
			\end{cases}.
		\end{equation*}
		Here $\mathcal{F}_{\eta}$ is the characteristic foliation of the pre-contact structure $\ker\big(\theta_C+\operatorname{pr}_{T\mathcal{F}}^{*}\eta\big)$.
	\end{defi}

	%\begin{remark}\label{rem:local-fol}
	%Given a regular coisotropic submanifold $C^{k+1}\subset(M^{2n+1},\xi)$, let us trivialize the line bundle $l$ using a local defining one-form $\alpha_C\in\Omega^{1}(C)$ for $\xi_C$. The space $\text{Def}^{U}_{\mathcal{F}}(C)$ then reduces to the space of sections $\eta\in\Gamma(V)$ of the local model \eqref{eq:locloc} such that
	%\begin{equation*}
	%\begin{cases}
	%\text{Graph}(\eta)\ \text{is coisotropic in}\ \big(V,\ker(p^{*}\alpha_C+\epsilon_{G}^{*}\theta_{can})\big)\\
	%\text{There is a foliated diffeomorphism}\\
	%\hspace{2cm} (C,\mathcal{F})\overset{\sim}{\rightarrow}\left(C,\ker\Big((\alpha_C+\epsilon_{G}(\eta))\wedge\big(d\alpha_C+d(\epsilon_{G}(\eta))\big)^{k-n}\Big)\right)
	%\end{cases}.
	%\end{equation*}
	%\end{remark}

	\subsection{First order deformations of coisotropic submanifolds and foliations}\label{subsec:first-order}
	
	\vspace{0.1cm}
	\noindent
	
	There are two classes of geometric structures whose deformation theory is important in this note, namely coisotropic submanifolds of contact manifolds and foliations. We will now recall the cochain complexes underlying first order deformations of these objects.

	\subsubsection{First order deformations of coisotropic submanifolds}
	Let $C^{k+1}\subset(M^{2n+1},\xi)$ be a regular coisotropic submanifold with characteristic foliation $\mathcal{F}$. First order deformations of $C$ as a coisotropic submanifold -- or equivalently, as a regular coisotropic submanifold (see Rem.~\ref{rem:transverse})-- are $1$-cocycles in the complex $\big(\Omega^{\bullet}(\mathcal{F},\ell),d_{\nabla^{\ell}}\big)$, which is defined as follows.
	Note that the line bundle $\ell=TC/\xi_C$ carries a flat $T\mathcal{F}$-connection $\nabla^{\ell}$, given by
	\begin{equation}\label{eq:conn}
		\nabla^{\ell}_{X}\theta_C(Y)=\theta_C([X,Y]),\hspace{0.5cm}X\in\Gamma(T\mathcal{F}),Y\in\mathfrak{X}(C),
	\end{equation}
	where $\theta_C\in\Omega^{1}(C,\ell)$ is the structure form of $\xi_C$. This connection induces a differential $d_{\nabla^{\ell}}$ on $\Gamma(\wedge^{\bullet}T^{*}\mathcal{F}\otimes \ell)$, defined by the usual Koszul formula
	\begin{align}\label{eq:diff}
		d_{\nabla^{\ell}}\eta(X_0,\ldots,X_m)&=\sum_{i=0}^{m}(-1)^{i}\nabla^{\ell}_{X_i}\big(\eta(X_0,\ldots,X_{i-1},\widehat{X_i},X_{i+1},\ldots,X_m)\big)\nonumber\\
		&\hspace{0.5cm}+\sum_{i<j}(-1)^{i+j}\eta\big([X_i,X_j],X_0,\ldots,\widehat{X_i},\ldots,\widehat{X_j},\ldots,X_m\big).
	\end{align}  
	The resulting complex $\big(\Omega^{\bullet}(\mathcal{F},\ell),d_{\nabla^{\ell}}\big)$ is the one that governs first order deformations of $C$. This follows in particular from the main result in \cite[\S 5]{jacobi}, which shows that coisotropic deformations of $C$ correspond with Maurer-Cartan elements of a certain $L_{\infty}$-algebra structure on $\Omega^{\bullet}(\mathcal{F},\ell)[1]$ whose unary bracket coincides with the differential $d_{\nabla^{\ell}}$. Below, we present a simple geometric proof for this fact which avoids the use of $L_{\infty}$-algebras.

	\begin{prop}
		\label{prop:first-order}
		Let $C^{k+1}\subset(M^{2n+1},\xi)$ be a regular coisotropic submanifold with characteristic foliation $\mathcal{F}$ and local model $(U,\ker\gamma_G)$. If $\eta_t$ is a smooth path of coisotropic sections of $U$ starting at the zero section $\eta_0=0$, then $\dt{\eta_0}$ is closed with respect to $d_{\nabla^{\ell}}$.
	\end{prop}

	\begin{proof}
		By Lemma~\ref{lem:coiso-in-model}, we get a smooth family of pre-contact forms $\theta_C+\operatorname{pr}_{T\mathcal{F}}^\ast\eta_t\in\Omega^{1}(C,\ell)$. Denoting the characteristic foliation of $\ker(\theta_C+\operatorname{pr}_{T\mathcal{F}}^\ast\eta_t)$ by $\mathcal{F}_{t}$, we have a smooth family of rank $2n-k$ distributions $T\mathcal{F}_t$ on $C$.
		Choose a local frame $X_1^t,\ldots,X_{2n-k}^t$ of $T\mathcal{F}_t$ that is smoothly depending on $t$. In particular, we have
		\begin{align}
			\label{eq:proof:prop:first-order:hyperplane}
			(\theta_C+\operatorname{pr}_{T\mathcal{F}}^\ast\eta_t)(X^t_i&)=0,\\
			\label{eq:proof:prop:first-order:characteristic}
			(\theta_C+\operatorname{pr}_{T\mathcal{F}}^\ast\eta_t)([X^t_i&,X^t_j])=0,
		\end{align}
		for all $i,j=1,\ldots,2n-k$.
		%is a smooth $1$-parameter family of regular $\ell$-valued precontact $1$-forms on $C$ of rank $2(k-n)$.
		%So, denoting by $\mathcal{F}_t$ the characteristic foliation of $(C,\theta_C+\operatorname{pr}_{T\mathcal{F}}^\ast \eta_t)$, one gets a smooth $1$-parameter family $T\mathcal{F}_t$ of distributions on $C$ of rank $2n-k$.
		%Choose a local frame $X_1^t,\ldots,X_{2n-k}^t$ of $T\mathcal{F}_t$ that is smoothly depending on $t$.
		%Then, one immediately gets that, in particular,
		%\begin{align}
		%\label{eq:proof:prop:first-order:hyperplane}
		%(\theta_C+\operatorname{pr}_{T\mathcal{F}}^\ast\eta_t)(X^t_i&)=0\in\Gamma(\ell),\\
		%\label{eq:proof:prop:first-order:characteristic}
		%(\theta_C+\operatorname{pr}_{T\mathcal{F}}^\ast\eta_t)([X^t_i&,X^t_j])=0\in\Gamma(\ell),
		%\end{align}
		%for all $i,j=1,\ldots,2n-k$.
		Differentiating the equation \eqref{eq:proof:prop:first-order:hyperplane} at time $t=0$ and using the hypothesis $\eta_0=0$, one gets that
		\begin{equation}
			\label{eq:proof:prop:first-order:hyperplane_diff}
			\theta_C(\dt{X}^0_i)+\dt{\eta}_0(X^0_i)=0,
		\end{equation}
		for all $i=1,\ldots,2n-k$.
		Similarly, differentiating the equation \eqref{eq:proof:prop:first-order:characteristic} at time $t=0$ and using $\eta_0=0$ as well as the equality \eqref{eq:proof:prop:first-order:hyperplane_diff}, we get
		\begin{align*}
			0&=\theta_C([\dt{X}^0_i,X^0_j])+\theta_C([X^0_i,\dt{X}^0_j])+\dt{\eta}_0([X^0_i,X^0_j])\\
			&=-\nabla^\ell_{X^0_j}(\theta_C(\dt{X}^0_i))+\nabla^\ell_{X^0_i}(\theta_C(\dt{X}^0_j))+\dt{\eta}_0([X^0_i,X^0_j])\\
			&=\nabla^\ell_{X^0_j}(\dt{\eta}_0(X^0_i))-\nabla^\ell_{X^0_i}(\dt{\eta}_0(X^0_j))+\dt{\eta}_0([X^0_i,X^0_j])\\
			&=(d_{\nabla^\ell}\dt{\eta}_0)(X^0_j,X^0_i),
		\end{align*}
		for all $i,j=1,\ldots,2n-k$.
		This proves that $d_{\nabla^\ell}\dt{\eta}_0$ vanishes.
	\end{proof}

	We now consider first order deformations of $C\subset(M,\xi)$ that are geometrically trivial, in the sense that they arise from applying a contact isotopy of $(M,\xi)$ to $C$. Recall that a contact isotopy is the flow of a time-dependent contact vector field on $(M,\xi)$.

	\begin{remark}\label{rem:contactVF}
		Given a contact manifold $(M,\xi)$ with quotient line bundle $L=TM/\xi$ and structure form $\theta\in\Omega^{1}(M,L)$, the space of contact vector fields $\mathfrak{X}_{\xi}(M)$ fits in a canonical direct sum decomposition \cite[Prop.~5.5]{jacobi}
		\[
		\mathfrak{X}(M)=\mathfrak{X}_{\xi}(M)\oplus\Gamma(\xi).
		\]
		It follows that there is an isomorphism of vector spaces
		\begin{equation}\label{eq:contact-vf2}
			X_{(-)}:\Gamma(L)\rightarrow\mathfrak{X}_{\xi}(M):\lambda\mapsto X_{\lambda},
		\end{equation}
		which is defined by setting $X_{\lambda}$ to be the unique contact vector field satisfying $\theta(X_{\lambda})=\lambda$.
	\end{remark}
	
	\color{black}
	It turns out that first order deformations of $C$ arising from contact isotopies are cohomologically trivial, i.e. one-coboundaries in the complex $\big(\Omega^{\bullet}(\mathcal{F},\ell),d_{\nabla^{\ell}}\big)$. This also follows from the main result in \cite{jacobi}, since the notion of gauge equivalence in the $L_{\infty}$-algebra structure on $\Omega^{\bullet}(\mathcal{F},\ell)[1]$ described there agrees with the equivalence relation by contact isotopies. Again, we believe it is instructive to give a geometric proof which avoids the use of $L_{\infty}$-algebras.

	\begin{prop}
		\label{prop:equiv}
		Let $C\subset(M,\xi)$ be a regular coisotropic submanifold with characteristic foliation $\mathcal{F}$ and local model $(U,\ker\gamma_G)$.
		If a smooth path $\eta_t$ of coisotropic sections of $U$, starting at $\eta_0=0$, is generated by a contact isotopy of $(U,\ker\gamma_G)$, then $\dt{\eta}_0$ is a $1$-coboundary in the complex $(\Omega^\bullet(\mathcal{F},\ell),d_{\nabla^\ell})$.
	\end{prop}
	
	\begin{proof}
		Let $\varphi_t$ be a contact isotopy of $(U,\ker\gamma_G)$ generating $\eta_t$, i.e.~
		$
		\varphi_t(C)=\eta_t(C).
		$
		It yields an isotopy $\psi_t:=p\circ\varphi_t|_C$ of $C$ satisfying
		\begin{equation}\label{eq:rel}
			\varphi_t|_C=\eta_t\circ\psi_t.
		\end{equation}
		We claim that $(\psi_t)_{*}$ matches the pre-contact distributions $\ker\theta_C$ and $\ker(\theta_C+\operatorname{pr}_{T\mathcal{F}}^\ast\eta_t)$.
		To prove this, note that
		\[
		\ker\theta_C=TC\cap\ker\gamma_G=\big((\varphi_t|_C)_{*}\big)^{-1}\ker\gamma_G,
		\]
		where we used that $\varphi_t$ is a contactomorphism. On the other hand,
		\[
		\ker(\theta_C+\operatorname{pr}_{T\mathcal{F}}^\ast\eta_t)=\big((\eta_t)_{*}\big)^{-1}\ker\gamma_G.
		\] 
		Along with the relation \eqref{eq:rel}, this proves the claim.
		%Consequently, for any $t\in\mathbb{R}$, the tangent map $T\psi_t$ maps, in particular, the precontact distribution $\ker\theta_C=(TC)\cap\ker\gamma_G=(T(\varphi_t|_C))^{-1}\ker\gamma_G\subseteq TC$ into the precontact distribution $\ker(\theta_C+\operatorname{pr}_{T\mathcal{F}}^\ast\eta_t)=(T\eta_t)^{-1}\ker\gamma_G\subseteq TC$.
		It follows that there is a unique isotopy $\Psi_t\colon\ell\to\ell$ of line bundle automorphisms covering $\psi_t\colon C\to C$, such that
		\begin{equation}
			\label{eq:proof:prop:equiv:pullback}
			\theta_C=\Psi_t^\ast(\theta_C+\operatorname{pr}_{T\mathcal{F}}^\ast\eta_t).
		\end{equation}
		
		Denote by $\Delta_t\in\Gamma(\mathcal{D}\ell)$ the time-dependent derivation generating $\Psi_t$; its symbol is the time dependent vector field $X_t:=\sigma(\Delta_t)\in\mathfrak{X}(C)$ generating $\psi_t$.
		This means in particular that, for all $\lambda\in\Gamma(\ell)$ and $Y\in\mathfrak{X}(C)$,
		\begin{equation*}
			\left.\frac{d}{dt}\right|_{t=0}\Psi_t^\ast\lambda=\Delta_0\lambda\quad\text{and}\quad \left.\frac{d}{dt}\right|_{t=0}\psi_t^\ast Y=[X_0,Y].
		\end{equation*}
		Consequently, we have for all $\alpha\in\Omega^1(C,\ell)$ and $Y\in\mathfrak{X}(C)$,
		\begin{align}
			\label{eq:proof:prop:equiv:Lie_derivative}
			\left.\frac{d}{dt}\right|_{t=0}(\Psi_t^\ast\alpha)(Y)=\left.\frac{d}{dt}\right|_{t=0}\Psi_t^\ast\left(\alpha\big((\psi_t^{-1})^{*}Y\big)\right)
			&=\Delta_0(\alpha(Y))-\alpha([X_0,Y]).
		\end{align}
		If we evaluate both sides of \eqref{eq:proof:prop:equiv:pullback} on arbitrary $Y\in\Gamma(T\mathcal{F})$ and then differentiate at time $t=0$ using \eqref{eq:proof:prop:equiv:Lie_derivative}, we get
		\begin{align*}
			0&=\left.\frac{d}{dt}\right|_{t=0}\theta_C(Y)\\
			&=\left.\frac{d}{dt}\right|_{t=0}(\Psi_t^\ast(\theta_C+\operatorname{pr}_{T\mathcal{F}}^\ast\eta_t))(Y)\\
			&=\left.\frac{d}{dt}\right|_{t=0}(\Psi_t^\ast\theta_C)(Y)+\left.\frac{d}{dt}\right|_{t=0}\operatorname{pr}_{T\mathcal{F}}^\ast\eta_t(Y)\\
			&=\Delta_0(\theta_C(Y))-\theta_C([X_0,Y]))+\dt{\eta}_0(Y)\\
			&=\nabla^\ell_Y(\theta_C(X_0))+\dt{\eta}_0(Y)\\
			&=(d_{\nabla^\ell}(\theta_C(X_0))+\dt{\eta}_0)(Y).
		\end{align*}
		This shows that $\dt{\eta}_0=-d_{\nabla^\ell}(\theta_C(X_0))$, i.e.~the $1$-cocycle $\dt{\eta}_0$ in $(\Omega^\bullet(\mathcal{F},\ell),d_{\nabla^\ell})$ is actually a coboundary, with primitive given by $-\theta_C(X_0)$. This finishes the proof.
	\end{proof}

	Prop.~\ref{prop:first-order} and Prop.~\ref{prop:equiv} motivate the following definition.
	
	\begin{defi}
		Let $C\subset(M,\xi)$ be regular coisotropic with characteristic foliation $\mathcal{F}$. 
		\begin{enumerate}[label=\roman*)]
			\item A \textbf{first order deformation} of $C$ is an element $\eta\in\Omega^{1}(\mathcal{F},\ell)$ that is $d_{\nabla^{\ell}}$-closed.
			\item We call $C$ \textbf{infinitesimally rigid} if the cohomology group $H^{1}(\mathcal{F},\ell)$ vanishes.
		\end{enumerate}
	\end{defi}
	
	In this note, we are interested in (un)obstructedness of first order deformations.
	
	\begin{defi}
		Let $C\subset(M,\xi)$ be a regular coisotropic submanifold with characteristic foliation $\mathcal{F}$ and local model $(U,\ker\gamma_G)$. A first order deformation $\eta$ of $C$ is said to be \textbf{unobstructed} if there exists a smooth path of coisotropic sections $\zeta_t\in\Gamma(U,\ker\gamma_G)$ satisfying $\zeta_0=0$ and $\dt{\zeta_{0}}=\eta$. Otherwise, $\eta$ is said to be \textbf{obstructed}.
	\end{defi}

	It is well-known that a regular coisotropic submanifold $C\subset(M,\xi)$ has obstructed first order deformations in general, see e.g. \cite[\S 3.1]{rigidity} for such an example. We will prove in this note that the coisotropic deformation problem of $C\subset(M,\xi)$ becomes unobstructed if one restricts to deformations whose characteristic foliation is diffeomorphic to $\mathcal{F}$.

	\subsubsection{First order deformations of foliations}
	Let $\mathcal{F}$ be a foliation on a manifold $C$. There is a flat $T\mathcal{F}$-connection on the normal bundle $N\mathcal{F}:=TC/T\mathcal{F}$, called the Bott connection, which is defined by
	\[
	\nabla^{N\mathcal{F}}_{X}(Y\ \text{mod}\ T\mathcal{F})=[X,Y]\ \text{mod}\ T\mathcal{F},\quad X\in\Gamma(T\mathcal{F}), Y\in\mathfrak{X}(C).
	\]
	We obtain a differential $d_{\nabla^{N\mathcal{F}}}$ on $\Gamma(\wedge^{\bullet}T^{*}\mathcal{F}\otimes N\mathcal{F})$ given by the same formula as \eqref{eq:diff}, and we denote by $\big(\Omega^{\bullet}(\mathcal{F},N\mathcal{F}),d_{\nabla^{N\mathcal{F}}}\big)$ the resulting complex. Heitsch \cite{heitsch} showed that this complex and its cohomology govern infinitesimal deformations of $\mathcal{F}$, in the following sense.
	
	Let $\mathcal{F}_{t}$ be a smooth path of foliations with $\mathcal{F}_{0}=\mathcal{F}$. Pick a complement $G$ to $T\mathcal{F}$ and identify $G\cong N\mathcal{F}$. %The induced Bott connection on $G$ is given by
	%\[
	%\nabla^{G}_{X}Y=\text{pr}_{G}[X,Y],\hspace{1cm} X\in\Gamma(T\mathcal{F}), Y\in\Gamma(G),
	%\]
	%where $\text{pr}_{G}:TC\rightarrow G$ is the projection.
	If $C$ is compact, then there exists $\epsilon>0$ such that $T\mathcal{F}_t$ remains transverse to $G$ for $0\leq t\leq\epsilon$. Hence, there exist $\eta_t\in\Gamma(T^{*}\mathcal{F}\otimes N\mathcal{F})$ such that
	\[
	T\mathcal{F}_t=\text{graph}(\eta_t)=\{X+\eta_t(X): X\in\Gamma(T\mathcal{F})\}.
	\]
	
	\begin{lemma}[\cite{heitsch}]
		\label{lem:first-order-fol}
		In the setup described above, we have:
		\begin{enumerate}
			\item The infinitesimal deformation $\dt{\eta_0}$ is closed with respect to $d_{\nabla^{N\mathcal{F}}}$.
			\item If the path $\mathcal{F}_t$ is generated by an isotopy $\phi_t$, i.e. $T\mathcal{F}_t=(\phi_t)_{*}T\mathcal{F}$, then the corresponding infinitesimal deformation is exact. Indeed,
			\[
			\dt{\eta_0}=d_{\nabla^{N\mathcal{F}}}(V_0\ \text{mod}\ T\mathcal{F}),
			\]
			where $V_t$ is the time-dependent vector field corresponding with the isotopy $\phi_t$.
		\end{enumerate}
	\end{lemma}

	Lemma \ref{lem:first-order-fol} motivates the following definition.
	
	\begin{defi}
		Let $C$ be a manifold with a foliation $\mathcal{F}$. 
		\begin{enumerate}[label=\roman*)]
			\item A \textbf{first order deformation} of $\mathcal{F}$ is an element $\eta\in\Omega^{1}(\mathcal{F},N\mathcal{F})$ that is $d_{\nabla^{N\mathcal{F}}}$-closed.
			\item We call $\mathcal{F}$ \textbf{infinitesimally rigid} if the cohomology group $H^{1}(\mathcal{F},N\mathcal{F})$ vanishes.
		\end{enumerate}
	\end{defi}

	\section{First order deformations with fixed characteristic foliation}\label{sec:two}
	
	In this section, we argue what are first order deformations when deforming a regular coisotropic submanifold $C$ inside $\text{Def}_{\mathcal{F}}(C)$. We show that first order deformations of $C$ give rise to trivial first order deformations of $\mathcal{F}$, via a chain map which we introduce below. While the construction of this chain map involves a choice of complement to the characteristic distribution $T\mathcal{F}$, the induced map in cohomology turns out to be completely canonical.
	
	\subsection{Smooth Paths in $\text{Def}_\mathcal{F}(C)$}
	\label{sec:smooth_paths}
	
	We first need to define a suitable notion of smoothness for paths in $\text{Def}_{\mathcal{F}}(C)$ and $\text{Def}^{U}_{\mathcal{F}}(C)$.
	
	\begin{defi}
		\label{def:smoothness}
		A path $C_t$ in $\text{Def}_{\mathcal{F}}(C)$ is smooth if there is a smooth path of embeddings $\Phi_t:C\hookrightarrow M$ such that $C_{t}=\Phi_t(C)$ and $\Phi_t:(C,\mathcal{F})\rightarrow (C_t,\mathcal{F}_{t})$ is a foliated diffeomorphism. 
	\end{defi}
	
	It will be important that a smooth path in  $\text{Def}_{\mathcal{F}}(C)$ comes with a smooth path of foliated diffeomorphisms to $(C,\mathcal{F})$. By passing to the contact thickening $(U,\ker\gamma_G)$ of $C\subset(M,\xi)$, Def.~\ref{def:smoothness} induces a notion of smoothness for paths in $\text{Def}^{U}_{\mathcal{F}}(C)$.
	
	\begin{defi}\label{def:smoothpath}
		A path $\eta_t$ in $\text{Def}^{U}_{\mathcal{F}}(C)$ is smooth if there is a smooth path of embeddings $\Phi_t:C\hookrightarrow U$ such that $\text{graph}(\eta_t)=\Phi_t(C)$ and $p\circ\Phi_t:(C,\mathcal{F})\rightarrow(C,\mathcal{F}_{\eta_{t}})$ is a foliated diffeomorphism. Here $\mathcal{F}_{\eta_{t}}$ is the foliation of the pre-contact structure $\ker(\theta_C+\operatorname{pr}_{T\mathcal{F}}^\ast\eta_t)$.
	\end{defi}

	\subsection{Description of first order deformations with fixed foliation}
	\label{sec:description_1st_order_deformations}
	
	To argue what are first order deformations when deforming a regular coisotropic submanifold $C$ within the class $\text{Def}_{\mathcal{F}}(C)$, we work in the local model $\text{Def}^{U}_{\mathcal{F}}(C)$ introduced in Def.~\ref{def:DefU}. We have to figure out what properties are satisfied by the infinitesimal deformation $\dt{\eta_0}$ coming from a smooth path $\eta_t$ in $\text{Def}^{U}_{\mathcal{F}}(C)$.
	One would expect that this infinitesimal deformation gives rise to a trivial infinitesimal deformation of the foliation $\mathcal{F}$, i.e.~a $1$-coboundary in the Bott complex $\big(\Omega^{\bullet}(\mathcal{F},N\mathcal{F}),d_{\nabla^{N\mathcal{F}}}\big)$.
	We will show that this does indeed happen, by means of a certain map $\Phi_G$ which is defined in terms of a complement $G$ to $T\mathcal{F}$.

	We first introduce some notation. In the sequel, we will freely use concepts related to the Atiyah algebroid and der-complex of a line bundle. See the Appendix for more details.	
	
	\vspace{0.2cm}
	\noindent\fbox{
		\parbox{\textwidth}{
			Let $C$ be any manifold.
			\begin{itemize}
				\item Assume that $\xi_C$ is a pre-contact distribution on $C$, with line bundle $\ell=TC/\xi_C$ and structure form $\theta_C\in\Omega^{1}(C,\ell)$.
				\item Let $\mathcal{F}$ be the characteristic foliation of $(C,\xi_C)$. Fix a complement $TC=T\mathcal{F}\oplus G$.
				\item The pre-contact form $\theta_C$ corresponds with a pre-symplectic Atiyah form $\varpi\in\Omega^{2}_{D}(\ell)$, see Prop.~\ref{prop:relation-luca}. Denote $K=\ker\varpi$. The symbol map $\sigma:D\ell\rightarrow TC$ restricts to a Lie algebroid isomorphism
				\[
				\sigma:K\overset{\sim}{\rightarrow}T\mathcal{F}.
				\] 
			\end{itemize}
		}
	}
	
	\vspace{0.2cm}

	Clearly, the chosen complement $G$ of $T\mathcal{F}$ gives rise canonically to a complement $\widetilde{G}$ of $K$.
	
	\begin{lemma}\label{lem:complement}
		The relation $\sigma(\widetilde{G})=G$ establishes a canonical 1-1 correspondence between:
		\begin{enumerate}[label=\roman*)]
			\item vector subbundles $G\subseteq TC$ such that $TC=T\mathcal{F}\oplus G$, 
			\item vector subbundles $\widetilde{G}\subseteq D\ell$ such that $D\ell=K\oplus\widetilde{G}$ and $\mathbbm{1}\in\Gamma(\widetilde{G})$.
		\end{enumerate}
	\end{lemma}
	
	%From now on we will assume to have chosen a complement $G$ to $T\mathcal{F}$ in $TC$, or equivalently the corresponding  complement $\widetilde{G}=\sigma^{-1}(G)$ to $K$ in $D\ell$ with $\mathbbm{1}\in\Gamma(\widetilde{G})$.
	The definition of the map $\Phi_G$ involves the following ingredients:
	\begin{itemize}
		\item The complement $\widetilde{G}$ yields a projection $\text{pr}_{K}:D\ell\rightarrow K$.
		\item The pre-symplectic Atiyah form $\varpi$ restricts to an isomorphism 
		\[
		\varpi^{\flat}|_{\widetilde{G}}:\widetilde{G}\rightarrow\big(\widetilde{G}\big)^{*}\otimes\ell.
		\]
		Hence, if $\beta\in N^{*}\mathcal{F}\otimes\ell\simeq G^{*}\otimes \ell$, then $\varpi^{\flat}|_{\widetilde{G}}^{-1}(\sigma^{*}\beta)$ is a well-defined element of $\widetilde{G}$.
	\end{itemize}

	\begin{defi}
		\label{def:Phi}
		We define 
		$
		\Phi_G\colon\Omega^\bullet(\mathcal{F},\ell)\longrightarrow\Omega^\bullet(\mathcal{F},N\mathcal{F})
		$
		by setting 
		\begin{equation}
			\label{eq:def:Phi}
			\left\langle\beta,\Phi_G(\eta)\big(\sigma(\Delta_1),\ldots,\sigma(\Delta_n)\big)\right\rangle=d_D(\operatorname{pr}_K^\ast\sigma^{*}\eta)\left(\Delta_1,\ldots,\Delta_n,\varpi^\flat|_{\widetilde{G}}^{-1}(\sigma^\ast\beta)\right).
		\end{equation}
		for $\eta\in\Omega^n(\mathcal{F},\ell), \beta\in\Gamma(N^{*}\mathcal{F}\otimes\ell)$ and $\Delta_1,\ldots,\Delta_n\in\Gamma(K)$.
	\end{defi}

	\begin{remark}
		\label{rem:Phi}
		The action of $\Phi_G$ on $\eta\in\Omega^n(\mathcal{F},\ell)$ can be equivalently described as follows.
		First apply the der-differential $d_D$ to $\operatorname{pr}_K^\ast\sigma^{*}\eta\in\Omega_D^n(\ell)$, which yields
		\begin{equation*}
			d_D(\operatorname{pr}_K^\ast\sigma^{*}\eta)\in\Gamma(\wedge^{n+1}K^\ast\otimes\ell)\oplus\Gamma(\wedge^n K^\ast\otimes\big(\widetilde{G}\big)^{*}\otimes\ell)\oplus\Gamma(\wedge^{n-1} K^\ast\otimes\wedge^2\big(\widetilde{G}\big)^{*}\otimes\ell).
		\end{equation*}
		Then pick the component in $\Gamma(\wedge^n K^\ast\otimes\big(\widetilde{G}\big)^{*}\otimes\ell)$ and apply to it the VB morphism
		\begin{equation*}
			-\operatorname{id}\otimes(\sigma\circ\varpi^\flat|_{\widetilde{G}}^{-1})\colon\wedge^\bullet K^\ast\otimes\big(\big(\widetilde{G}\big)^{*}\otimes\ell\big)\longrightarrow\wedge^\bullet K^\ast\otimes G.
		\end{equation*}
		The result is the element $\sigma^{*}\left(\Phi_G(\eta)\right)\in\Omega^{n}(K,N\mathcal{F})$, which determines $\Phi_G(\eta)\in\Omega^{n}(\mathcal{F},N\mathcal{F})$.
	\end{remark}

	\begin{thm}
		\label{thm:first-order}
		Let $C\subset(M,\xi)$ be a compact regular coisotropic submanifold with characteristic foliation $\mathcal{F}$.
		A choice of splitting $TC=T\mathcal{F}\oplus G$ yields the local model $(U,\ker\gamma_G)$.
		Assume that $\eta_t\in\Gamma(U)$ is a smooth path of coisotropic sections, with $\eta_0=0$.
		%in $\text{Def}_{\mathcal{F}}(C)$, passing through $C$ at $t=0$, and denote by $\mathcal{F}_t$ the characteristic foliation of the induced regular precontact form $\theta_C+\operatorname{pr}_{T\mathcal{F}}^\ast\eta_t$.
		%Then 
		\begin{enumerate}[label=(\arabic*)]
			\item
			\label{enumitem:thm:first-order:1}
			The infinitesimal deformation $\dt{\eta_0}$ is closed in $(\Omega^\bullet(\mathcal{F},\ell),d_{\nabla^\ell})$.
			\item 
			\label{enumitem:thm:first-order:2}
			Let $\mathcal{F}_t$ be the characteristic foliation of the pre-contact form $\theta_t:=\theta_C+\operatorname{pr}_{T\mathcal{F}}^\ast\eta_t$. For small enough $t$, there is a smooth path $\zeta_t\in\Omega^{1}(\mathcal{F},N\mathcal{F})$ such that $T\mathcal{F}_t=\text{graph}(\zeta_t)$. Moreover,
			\[
			\Phi_G(\dt\eta_0)=\dt\zeta_0.
			\]
			%If $C$ is compact, then there exists (at least for small $t$) a smooth path $\alpha_t$ in $\Omega^1(\mathcal{F},N\mathcal{F})$, with $\alpha_0=0$, such that $T\mathcal{F}_t=\text{graph}(\alpha_t)$ and moreover
			%\begin{equation*}
			%\Phi_G(\dot\eta_0)=\dot\alpha_0.
			%\end{equation*}
			\item
			\label{enumitem:thm:first-order:3}
			If $\eta_t$ is a smooth path in $\text{Def}^U_{\mathcal{F}}(C)$, then $\Phi_G(\dt{\eta_0})$ is exact in $\big(\Omega^{\bullet}(\mathcal{F},N\mathcal{F}),d_{\nabla^{N\mathcal{F}}}\big)$.
		\end{enumerate}
	\end{thm}

	\begin{proof}
		\ref{enumitem:thm:first-order:1}
		This holds by Prop.~\ref{prop:first-order}.
		
		\ref{enumitem:thm:first-order:2}
		The pre-contact form $\theta_t$  corresponds to a pre-symplectic Atiyah form $\varpi_t$, given by
		\begin{equation}\label{eq:varpi}
			\varpi_t:=d_D(\sigma^\ast\theta_t)=\varpi+d_D((\sigma\circ\operatorname{pr}_K)^\ast\eta_t).
		\end{equation}
		We set $K_t:=\ker\varpi_t$.
		Then the following conditions are equivalent:
		\begin{enumerate}[label=(\alph*)]
			\item
			\label{enumitem:proof:thm:first-order:2a}
			$T\mathcal{F}_t$ is transverse to $G$, i.e.~there exists $\zeta_t\in\Omega^1(\mathcal{F},G)$ such that $T\mathcal{F}_t=\text{graph}(\zeta_t)$, 
			\item
			\label{enumitem:proof:thm:first-order:2b}
			$K_t$ is transverse to $\widetilde{G}$, i.e.~there exists $\mu_t\in\Omega^1(K,\widetilde{G})$ such that $K_t=\text{graph}(\mu_t)$.
		\end{enumerate}
		Compactness of $C$ implies that there exists $\epsilon>0$ such that the equivalent conditions~\ref{enumitem:proof:thm:first-order:2a} and~\ref{enumitem:proof:thm:first-order:2b} hold for $-\epsilon<t<\epsilon$.
		In this case, $\zeta_t\in\Omega^1(\mathcal{F},G)$ and $\mu_t\in\Omega^1(K,\widetilde{G})$ are related by %$\zeta_t\in\Omega^1(\mathcal{F},N\mathcal{F})$ and the corresponding $\widetilde{\zeta}_t\in\Omega^1(K,D_\perp\ell)$ depend  smoothly on $t$, satisfy $\zeta_0=\widetilde{\zeta}_0=0$ and are related by
		\begin{equation}\label{eq:relation}
			\zeta_t\circ\sigma|_K=\sigma\circ\mu_t.
			%\in\Omega^1(K,N\mathcal{F}),\quad\text{for all}\ -\epsilon<t<\epsilon.
		\end{equation}
		We now compute for $\Delta\in\Gamma(K)$,
		\begin{equation*}
			0=\varpi_t^\flat(\Delta+\mu_t(\Delta))=\varpi^\flat(\mu_t(\Delta))+(d_D((\sigma\circ\operatorname{pr}_K)^\ast\eta_t))^\flat(\Delta+\mu_t(\Delta)),
		\end{equation*}
		where we used \eqref{eq:varpi} and the fact that $K=\ker\varpi$.
		Differentiating at time $t=0$ and using that $\eta_0=0$ and $\mu_0=0$, we get 
		\begin{equation*}
			0=\varpi^\flat(\dt{\mu_0}(\Delta))+(d_D((\sigma\circ\operatorname{pr}_K)^\ast\dt{\eta}_0))^\flat(\Delta).
		\end{equation*}
		With Def.~\ref{def:Phi} in mind, this implies that for any $\Delta\in\Gamma(K)$ and $\beta\in\Gamma(G^{*}\otimes\ell)$,
		\begin{align*}
			0&=\left\langle\varpi^\flat(\dt{\mu_0}(\Delta)),\varpi^\flat|_{\widetilde{G}}^{-1}(\sigma^\ast\beta)\right\rangle+\left\langle(d_D((\sigma\circ\operatorname{pr}_K)^\ast\dt{\eta}_0))^\flat(\Delta),\varpi^\flat|_{\widetilde{G}}^{-1}(\sigma^\ast\beta)\right\rangle\\
			&=-\left\langle\sigma^\ast\beta,\dt{\mu_0}(\Delta)\right\rangle+\left\langle\beta,\Phi_{G}(\dt{\eta}_0)(\sigma(\Delta))\right\rangle\\
			&=-\left\langle\beta,\sigma(\dt{\mu_0}(\Delta))\right\rangle+\left\langle\beta,\Phi_{G}(\dt{\eta}_0)(\sigma(\Delta))\right\rangle\\
			&=-\big\langle\beta,\dt{\zeta_0}(\sigma(\Delta))\big\rangle+\left\langle\beta,\Phi_{G}(\dt{\eta}_0)(\sigma(\Delta))\right\rangle\\
			&=\left\langle\beta,\big(\Phi_G(\dt\eta_0)-\dt\zeta_0\big)(\sigma(\Delta))\right\rangle,
		\end{align*}
		using \eqref{eq:relation} in the fourth equality.
		This shows that $\Phi_G(\dt{\eta}_0)=\dt\zeta_0$.
		
		\ref{enumitem:thm:first-order:3} 
		By part $(2)$, we only have to show that $\dt\zeta_0$ is exact in $(\Omega^\bullet(\mathcal{F},N\mathcal{F}),d_{\nabla^{N\mathcal{F}}})$.
		Since $\eta_t$ is a smooth path in $\text{Def}_\mathcal{F}^U(C)$, according to Def.~\ref{def:smoothpath} there exists a smooth family of foliated diffeomorphisms
		\begin{equation*}
			\phi_t\colon(C,T\mathcal{F})\longrightarrow(C,T\mathcal{F}_t).
		\end{equation*}
		Precomposing $\phi_t$ with $\phi_0^{-1}$, we can assume that $\phi_0=\operatorname{id}_C$, so that the smooth family of foliations $\mathcal{F}_t$ is generated by the isotopy $\phi_t$. As $T\mathcal{F}_t=\operatorname{graph}(\zeta_t)$, Lemma~\ref{lem:first-order-fol} implies that $\dt{\zeta}_0$ is exact in $(\Omega^\bullet(\mathcal{F},N\mathcal{F}),d_{\nabla^{N\mathcal{F}}})$. This finishes the proof.
	\end{proof}
	
	Thm.~\ref{thm:first-order} motivates the following (provisional) definition.
	
	\begin{defi}[Provisional]
		\label{def:first-order-defs}
		Let $C\subset(M,\xi)$ be a compact regular coisotropic submanifold with characteristic foliation $\mathcal{F}$.
		When deforming $C$ inside $\text{Def}_{\mathcal{F}}(C)$, a \textbf{first order deformation} of $C$ is an element $\eta\in\Omega^{1}(\mathcal{F},\ell)$ such that
		\begin{equation*}
			d_{\nabla^{\ell}}\eta=0\quad \text{and}\quad \Phi_G(\eta)\ \text{is exact in}\ (\Omega^\bullet(\mathcal{F},N\mathcal{F}),d_{\nabla^{N\mathcal{F}}}).
		\end{equation*}
	\end{defi}
	
	In what follows, we study the map $\Phi_G$ in more detail.
	Section~\ref{sec:chain_map} gives an alternative description for it, which will allow us to prove that $\Phi_G$ is a chain map from $(\Omega^\bullet(\mathcal{F},\ell),d_{\nabla^\ell})$ to $(\Omega^\bullet(\mathcal{F},N\mathcal{F}),d_{\nabla^{N\mathcal{F}}})$.
	In Section~\ref{sec:induced_map_cohomology}, we show that the induced map in cohomology is canonical, i.e.~$[\Phi_G]\colon H^\bullet(\mathcal{F},\ell)\to H^\bullet(\mathcal{F},N\mathcal{F})$ does not depend on the complement $G$. In particular, this yields an invariant description of first order deformations which refines Def.~\ref{def:first-order-defs}.

	\subsection{A chain map}
	\label{sec:chain_map}
		Assume the setup described in the box at the beginning of \S\ref{sec:description_1st_order_deformations}. To show that $\Phi_G:(\Omega^\bullet(\mathcal{F},\ell),d_{\nabla^\ell})\to(\Omega^\bullet(\mathcal{F},N\mathcal{F}),d_{\nabla^{N\mathcal{F}}})$ is a chain map, we proceed in 2 steps:
		\begin{enumerate}
			\item Construct a canonical chain map 
			\begin{equation}\label{eq:step1}
				\left(\Omega^\bullet(K,J^1_\perp\ell),d_{\nabla^{J^1_\perp\ell}}\right)\longrightarrow\big(\Omega^\bullet(\mathcal{F},N\mathcal{F}),d_{\nabla^{N\mathcal{F}}}\big).
			\end{equation}
			Here $J^{1}_{\perp}\ell$ is a suitable subbundle of $J^{1}\ell$ to be introduced later.
			\item Construct a chain map
			\begin{equation}\label{eq:step2}
				\big(\Omega^{\bullet}(\mathcal{F},\ell),d_{\nabla^{\ell}}\big)\rightarrow\left(\Omega^\bullet(K,J^1_\perp\ell),d_{\nabla^{J^1_\perp\ell}}\right)
			\end{equation}
			which depends on the choice of complement $G$.
		\end{enumerate}
		We then show that $\Phi_G$ is the composition of these two maps, hence it is also a chain map.

	\subsubsection{Step 1 in construction of $\Phi_G$}
	%Let $C\subset(M,\xi)$ be a regular coisotropic submanifold with characteristic foliation $\mathcal{F}$ and quotient line bundle $\ell=TC/\xi_C$. Recall from \S\ref{subsec:first-order} that first order deformations of $C$ as a coisotropic submanifold are $1$-cocycles in the complex $\big(\Omega^{\bullet}(\mathcal{F},l),d_{\nabla^{l}}\big)$, while first order deformations of the foliation $\mathcal{F}$ are $1$-cocycles in $\big(\Omega^{\bullet}(\mathcal{F},N\mathcal{F}),d_{\nabla^{N\mathcal{F}}}\big)$.
	%In this subsection, we show that for each choice of complement $G$ to $T\mathcal{F}$, the associated $\Phi_G$ yields a chain map between these complexes.
	%Let us start recalling what is our setup.
	%Assume to have a manifold $C$ and a line bundle $\ell\to C$.
	%Equip $C$ with a regular $\ell$-valued precontact form $\theta\in\Omega^1(C,\ell)$ or, equivalently, with the regular $\ell$-valued presymplectic Atiyah form $\varpi\in\Omega^2_D(\ell)$ which corresponds bijectively to $\theta$ via the relation $\varpi=\mathrm{d}_D(\sigma^\ast\theta)$.
	%Denote by $\mathcal{F}$ the charateristic foliation of $(C,\theta)$ and by $K\subset D\ell$ the kernel of $\varpi$ so that the symbol map $\sigma\colon D\ell\to TC$ induces a Lie algebroid isomorphism $K\overset{\sim}{\longrightarrow}T\mathcal{F},\ \Delta\longmapsto\sigma(\Delta)$.
	
	To construct the map \eqref{eq:step1}, we start by remarking that the Lie algebroid isomorphism $\sigma:K\rightarrow T\mathcal{F}$ relates some natural representations carried by these Lie algebroids.
	On one hand,  $T\mathcal{F}$ carries the following representations.
	%Introduce the normal bundle $N\mathcal{F}:=TC/T\mathcal{F}$, with quotient map $p\colon TC\to N\mathcal{F}$, so that the conormal bundle $N^\ast\mathcal{F}$ gets identified with $T^\circ\mathcal{F}\subset T^\ast C$, the annihilator of $T\mathcal{F}$, via the following VB isomorphism, covering $\text{id}_C\colon C\to C$,
	%\begin{equation*}
	%N^\ast\mathcal{F}\overset{\sim}{\longrightarrow} T^\circ\mathcal{F},\ \alpha_x\longmapsto\alpha_x\circ p_x.
	%\end{equation*}
	%Then recall that the Lie algebroid $T\mathcal{F}\Rightarrow C$ admits the following natural representations. 
	\begin{itemize}
		\item The $T\mathcal{F}$-representation $\nabla^\ell$ on $\ell$ given by
		\begin{equation*}
			\nabla^\ell_X(\theta_C(Y))=\theta_C([X,Y]),\quad\text{for}\ X\in\Gamma(T\mathcal{F})\ \text{and}\ Y\in\mathfrak{X}(C).
		\end{equation*}
		Alternatively, $\nabla^\ell$ can be characterized by the following commutative diagram
		\begin{equation}
			\label{eq:nabla^ell}
			\begin{tikzcd}
				K\arrow[rr, hook, "\text{incl}"]\arrow[d, hook, two heads, swap, "\sigma"]&& D\ell\\
				T\mathcal{F}\arrow[urr, dashed, swap, "\nabla^\ell"]&&
			\end{tikzcd}.
		\end{equation}
		To see this, we compute for $\Delta\in\Gamma(K)$ and $\square\in\Gamma(D\ell)$,
		%Indeed, using $\varpi=\mathrm{d}_D(\sigma^\ast\theta)$ and $T\mathcal{F}\subset\ker\theta$, for any $\Delta\in\Gamma(K)$ and $\square\in\mathcal{D}\ell$, one can compute
		\begin{align*}
			0&=\varpi(\Delta,\square)\\
			&=(d_D(\sigma^\ast\theta_C))(\Delta,\square)\\
			&=\Delta(\theta_C(\sigma(\square)))-\cancel{\square(\theta_C(\sigma(\Delta)))}-\theta_C([\sigma(\Delta),\sigma(\square)]).
		\end{align*}
		\item The Bott representation $\nabla^{N\mathcal{F}}$ of $T\mathcal{F}$ on $N\mathcal{F}$ given by
		\begin{equation*}
			\nabla^{N\mathcal{F}}_X(Y\ \text{mod}\ T\mathcal{F})=[X,Y]\ \text{mod}\ T\mathcal{F},\quad\text{for}\ X\in\Gamma(T\mathcal{F}),\ Y\in\mathfrak{X}(C).
		\end{equation*} 
		%\item The dual Bott representation $\nabla^{N^\ast\mathcal{F}}$ of $T\mathcal{F}$ on $N^\ast\mathcal{F}$, which is defined by 
		%\begin{equation*}
		%\left\langle\nabla^{N^\ast\mathcal{F}}_X\alpha,Y\ \text{mod}\ T\mathcal{F}\right\rangle=X\langle\alpha,Y\ \text{mod}\ T\mathcal{F}\rangle-\left\langle\alpha,\nabla^{N\mathcal{F}}_X(Y\ \text{mod}\ T\mathcal{F})\right\rangle.
		%\end{equation*}
		%for $X\in\Gamma(T\mathcal{F}),\alpha\in\Gamma(N^\ast\mathcal{F})$ and $Y\in\mathfrak{X}(C)$. Viewing $\Gamma(N^\ast\mathcal{F})$ as $\Gamma(T^\circ\mathcal{F})\subset\Omega^1(C)$, the connection $\nabla^{N^\ast\mathcal{F}}$ is simply given by
		%Unravelling the definitions of $\nabla^\ell$ and $\nabla^{N\mathcal{F}}$ on the right hand side, it turns out that
		%\begin{equation*}
		%\nabla_X^{N^\ast\mathcal{F}}\alpha=\mathcal{L}_X\alpha.
		%\end{equation*}
		%for all $\alpha\in\Gamma(N^\ast\mathcal{F})\simeq\Gamma(T^\circ\mathcal{F})\subset\Omega^1(C)$ and $X\in\Gamma(T\mathcal{F})$.
	\end{itemize}
	
	The Lie algebroid $K$ carries representations on $\ell$ and on vector bundles $D_\perp\ell$ and $J^1_\perp\ell$ which we now introduce. We first define $D_\perp\ell:=D\ell/K$. Since $\sigma:D\ell\rightarrow TC$ maps $K$ bijectively onto $T\mathcal{F}$, it descends to a VB morphism
	\begin{equation*}
		\sigma_\perp\colon D_\perp\ell\longrightarrow N\mathcal{F},\ \Delta\ \text{mod}\ K\longmapsto\sigma(\Delta)\ \text{mod}\ T\mathcal{F}.
	\end{equation*}
	Next, we define $J^1_\perp\ell\subset J^1\ell\simeq(D\ell)^\ast\otimes\ell$ as the annihilator of $K\subset D\ell$, i.e.
	\begin{equation*}
		J^1_\perp\ell:=\sqcup_{x\in C}\{j^1_x\lambda\,\mid\lambda\in\Gamma(\ell)\ \text{such that}\ \delta(\lambda)=0,\ \text{for all}\ \delta\in K_x\}.
	\end{equation*}
	Note that we have an identification
	\begin{equation}\label{eq:iden}
		(D_\perp\ell)^\ast\otimes\ell\overset{\sim}{\longrightarrow} J^1_\perp\ell.
	\end{equation}
	The Lie algebroid $K$ now carries the following natural representations.
	\begin{itemize}
		\item The tautological representation of $D\ell$ on $\ell$ restricts to a representation of $K$, namely
		\[
		\nabla^{\ell}:K\rightarrow D\ell:\Delta\mapsto\Delta.
		\] 
		Because of the commutative diagram~\eqref{eq:nabla^ell}, using the symbol $\nabla^\ell$ to denote both the $T\mathcal{F}$- and $K$-representation on $\ell$ should not lead to any confusion. 
		\item The Bott representation $\nabla^{D_\perp\ell}$ of $K$ on $D_\perp\ell$ given by
		\begin{equation*}
			\nabla^{D_\perp\ell}_\Delta(\square\ \text{mod}\ K)=[\Delta,\square]\ \text{mod}\ K,\quad\text{for}\ \Delta\in\Gamma(K),\ \square\in\Gamma(D\ell).
		\end{equation*} 
		\item The $\ell$-twisted dual Bott representation of $K$ on $J^1_\perp\ell$, which is defined by
		\begin{equation}
			\label{eq:K_representation_J^1_perp ell}
			\left\langle\nabla^{J^1_\perp\ell}_\Delta\alpha,\square\ \text{mod}\ K\right\rangle=\nabla_\Delta^\ell\langle\alpha,\square\ \text{mod}\ K\rangle-\left\langle\alpha,\nabla^{D_\perp\ell}_\Delta(\square\ \text{mod}\ K)\right\rangle,
		\end{equation}
		for $\Delta\in\Gamma(K), \alpha\in\Gamma(J^1_\perp\ell)$ and $\square\in\Gamma(D\ell)$. Here we use the identification \eqref{eq:iden}. Viewing sections of $J^1_\perp\ell$ as elements of $\Gamma(K^{\circ}\otimes\ell)\subset\Gamma((D\ell)^{*}\otimes\ell)$, it follows from Cartan calculus with Atiyah forms $\Omega_D^\bullet(\ell)$ that  
		\begin{equation*}
			\nabla_\Delta^{J^1_\perp\ell}\alpha=\mathcal{L}_\Delta\alpha.
		\end{equation*}
	\end{itemize}
	
	It is clear that under the Lie algebroid isomorphism  $\sigma:K\overset{\sim}{\rightarrow}T\mathcal{F}$, the representations of $T\mathcal{F}$ and $K$ recalled above are related as follows.
	\begin{itemize}%[leftmargin=0.3in]
		\item The identity map $\text{id}_\ell\colon\ell\to\ell$ intertwines the representations of $K$ and $T\mathcal{F}$ on $\ell$, i.e.
		\begin{equation*}
			\nabla^\ell_\Delta\lambda=\nabla^\ell_{\sigma(\Delta)}\lambda,\quad\text{for}\ \Delta\in\Gamma(K)\ \text{and}\ \lambda\in\Gamma(\ell).
		\end{equation*}
		\item The VB morphism $\sigma_\perp\colon D_\perp\ell\to N\mathcal{F}$ intertwines the $K$-representation on $D_\perp\ell$ and the $T\mathcal{F}$-representation on $N\mathcal{F}$, i.e. for $\Delta\in\Gamma(K)$ and $\square\in\Gamma(D\ell)$,
		\begin{equation}\label{eq:comp-connections}
			\sigma_\perp(\nabla^{D_\perp\ell}_\Delta(\square\ \text{mod}\ K))=\nabla^{N\mathcal{F}}_{\sigma(\Delta)}(\sigma(\square)\ \text{mod}\ T\mathcal{F}).
		\end{equation}
		%\item The VB morphism $\sigma_\perp^\ast\colon N^\ast\mathcal{F}\to J^1_\perp\ell$ intertwines the representations of $T\mathcal{F}$ and $K$ on $N^\ast\mathcal{F}$ and $J^1_\perp\ell$ respectively, i.e., for all $\Delta\in\Gamma(K)$ and $\alpha\in N^\ast\mathcal{F}$,
		%\begin{equation*}
		%\nabla^{J^1_\perp\ell}_\Delta(\sigma_\perp^\ast(\alpha))=\sigma_\perp^\ast(\nabla^{N\mathcal{F}}_{\sigma(\Delta)}(\alpha)).
		%\end{equation*}
	\end{itemize}
	
	The $K$-representations on $D_\perp\ell$ and $J^1_\perp\ell$ are also related, as we show now. Recall that we are given an $\ell$-valued pre-symplectic Atiyah form $\varpi\in\Gamma(\wedge^{2}(D\ell)^{*}\otimes\ell)$ with $\ker\varpi^{\flat}=K$. It follows that 
	\[
	J^1_\perp\ell=\operatorname{im}\big\{\varpi^\flat\colon D\ell\to (D\ell)^\ast\otimes\ell\simeq J^1\ell\big\},
	\]
	and moreover, $\varpi$ descends to a non-degenerate form $\varpi_\perp\in\Gamma(\wedge^2(D_\perp\ell)^\ast\otimes\ell)$ which makes the following diagram commute
	\begin{equation*}
		\begin{tikzcd}
			D\ell\arrow[rr, "\varpi^\flat"]\arrow[d, two heads]&&J^1\ell\\
			D_\perp\ell\arrow[rr, hook, two heads, dashed, swap, "\varpi_\perp^\flat"]&&J^1_\perp\ell\arrow[u, hook, swap, "\text{incl}"].
		\end{tikzcd}.
	\end{equation*}
	We will denote the inverse of $\varpi^\flat_\perp\colon D_\perp\ell\overset{\sim}{\longrightarrow} J^1_\perp\ell$ by $\varpi_\perp^\sharp\colon J^1_\perp\ell\overset{\sim}{\longrightarrow} D_\perp\ell$.

	\begin{prop}
		\label{prop:sigma_perp_omega_perp}
		The VB isomorphism $\varpi_\perp^\flat\colon D_\perp\ell\to J^1_\perp\ell$ intertwines the $K$-representations on $D_\perp\ell$ and $J^1_\perp\ell$, i.e. for $\Delta\in\Gamma(K)$ and $\square\in\Gamma(D\ell)$ we have
		\begin{equation*}
			\varpi_\perp^\flat(\nabla^{D_\perp\ell}_\Delta(\square\ \text{mod}\ K))=\nabla^{J^1_\perp\ell}_\Delta(\varpi_\perp^\flat(\square\ \text{mod}\ K)).
		\end{equation*}
	\end{prop}
	
	\begin{proof}
		Since $d_D\varpi=0$ and $K=\ker\varpi$, the Cartan formula %$\mathcal{L}_\Delta=[\iota_\Delta,\mathrm{d}_D]$
		implies that for $\Delta\in\Gamma(K)$, $$\mathcal{L}_\Delta\varpi=\iota_\Delta d_D\varpi+d_D\iota_\Delta\varpi=0.$$
		Hence, we can compute for any $\Delta\in\Gamma(K)$ and $\square,\square'\in\Gamma(D\ell)$, 
		%		\begin{align*}
		%			0&=(\mathrm{d}_D\varpi)(\Delta,\square,\square')\\
		%			&=\Delta(\varpi(\square,\square'))-\square(\varpi(\Delta,\square'))+\square'(\varpi(\Delta,\square))\\
		%			&\phantom{=}-\varpi([\Delta,\square],\square')+\varpi([\Delta,\square'],\square)-\varpi([\square,\square'],\Delta)
		%		\end{align*}
		\begin{align*}
			0&=(\mathcal{L}_\Delta\varpi)(\square,\square')\\
			&=\Delta(\varpi(\square,\square'))-\varpi([\Delta,\square],\square')-\varpi(\square,[\Delta,\square'])\\
			&=\nabla^\ell_\Delta\left\langle\varpi_\perp^\flat(\square\ \text{mod}\ K),\square'\ \text{mod}\ K\right\rangle-\left\langle\varpi_\perp^\flat\big(\nabla^{D_\perp\ell}_\Delta(\square\ \text{mod}\ K)\big),\square'\ \text{mod}\ K\right\rangle\\
			&\phantom{=}-\left\langle\varpi_\perp^\flat(\square\ \text{mod}\ K),\nabla^{D_\perp\ell}_\Delta(\square'\ \text{mod}\ K)\right\rangle\\
			&=\left\langle\nabla^{J^1_\perp\ell}_\Delta(\varpi_\perp^\flat(\square\ \text{mod}\ K)),\square'\ \text{mod}\ K\right\rangle-\left\langle\varpi_\perp^\flat\big(\nabla^{D_\perp\ell}_\Delta(\square\ \text{mod}\ K)\big),\square'\ \text{mod}\ K\right\rangle.
		\end{align*}
		This shows that $$\nabla^{J^1_\perp\ell}_\Delta(\varpi_\perp^\flat(\square\ \text{mod}\ K))=\varpi_\perp^\flat\big(\nabla^{D_\perp\ell}_\Delta(\square\ \text{mod}\ K)\big)$$ for all $\Delta\in\Gamma(K)$ and $\square\in\Gamma(D\ell)$, as we wanted to prove.
	\end{proof}
	
	Combining the result \eqref{eq:comp-connections} and Prop.~\ref{prop:sigma_perp_omega_perp}, we obtain the following.

	\begin{cor}
		\label{cor:sigma_perp_omega_perp}
		Under the Lie algebroid isomorphism $\sigma:K\overset{\sim}{\to} T\mathcal{F}$, the VB morphism $\sigma_\perp\circ\varpi_\perp^\sharp\colon J^1_\perp\ell\to N\mathcal{F}$ intertwines the $K$-representation on $J^1_\perp\ell$ and the $T\mathcal{F}$-representation on $N\mathcal{F}$. That is, for any $\Delta\in\Gamma(K)$ and $\alpha\in\Gamma(J^1_\perp\ell)$,
		\begin{equation*}
			(\sigma_\perp\circ\varpi_\perp^\sharp)\big(\nabla^{J^1_\perp\ell}_\Delta\alpha\big)=\nabla^{N\mathcal{F}}_{\sigma(\Delta)}((\sigma_\perp\circ\varpi_\perp^\sharp)\alpha).
		\end{equation*}
	\end{cor}

	Cor.~\ref{cor:sigma_perp_omega_perp} yields a canonical chain map which is the first step in the construction of $\Phi_G$.
	
	\begin{cor}\label{cor:chain1}
		The VB morphism $$\wedge^\bullet(\sigma|_K^{-1})^{*}\otimes(\sigma_\perp\circ\varpi_\perp^\sharp)\colon\wedge^\bullet K^\ast\otimes J^1_\perp\ell\longrightarrow\wedge^\bullet T^\ast\mathcal{F}\otimes N\mathcal{F}$$ 
		induces a chain map at the level of sections
		\begin{equation*}
			\left(\Omega^\bullet(K,J^1_\perp\ell),d_{\nabla^{J^1_\perp\ell}}\right)\longrightarrow\left(\Omega^\bullet(\mathcal{F},N\mathcal{F}),d_{\nabla^{N\mathcal{F}}}\right).
		\end{equation*}
	\end{cor}

	\subsubsection{Step 2 in construction of $\Phi_G$}
	The construction of the map \eqref{eq:step2} requires a choice of complement $TC=T\mathcal{F}\oplus G$. By Lemma~\ref{lem:complement}, there is an induced complement $D\ell=K\oplus\widetilde{G}$. Denoting by $\widetilde{G}^\dagger:=(\widetilde{G})^{*}\otimes\ell$ the $\ell$-twisted dual of $\widetilde{G}$, we will make the following identifications:
	\[
	D_\perp\ell\simeq\widetilde{G}\quad\text{and}\quad  J^1_\perp\ell\simeq \widetilde{G}^\dagger.
	\]
	The VB isomorphism 
	\begin{equation*}
		\wedge^i(D\ell)^\ast\otimes\ell\simeq\bigoplus_{a+b=i}\wedge^a(D_\perp\ell)^\ast\otimes\wedge^b K^\ast\otimes\ell
	\end{equation*}
	induces an isomorphism of $C^\infty(C)$-modules
	\begin{equation}\label{eq:bigrading}
		\Omega_D^i(\ell)\simeq\bigoplus_{a+b=i}\Omega^{a,b}_D(\ell),\quad\text{where}\ \ \Omega_D^{a,b}(\ell):=\Gamma(\wedge^a(D_\perp\ell)^\ast\otimes\wedge^b K^\ast\otimes\ell).
	\end{equation}
	Since $K\subset D\ell$ is involutive, the Koszul formula for the der-differential $d_D$ implies that
	\begin{equation*}
		d_D\Omega_D^{a,b}(\ell)\subseteq\Omega_D^{a,b+1}(\ell)\oplus\Omega_D^{a+1,b}(\ell)\oplus\Omega_D^{a+2,b-1}(\ell). 
	\end{equation*}
	Therefore, $d_D$ decomposes as
	\begin{equation}
		\label{eq:decomposition_d_D}
		d_D=d_{0,1}^G+d_{1,0}^G+d_{2,-1}^G,
	\end{equation}
	where each summand is bi-homogeneous with respect to the bi-grading $\Omega^{\bullet,\bullet}_D(\ell)$, i.e.
	\begin{equation*}
		d^{G}_{0,1}\colon\Omega_D^{\bullet,\bullet}(\ell)\to\Omega_D^{\bullet,\bullet+1}(\ell),\quad d^{G}_{1,0}\colon\Omega_D^{\bullet,\bullet}(\ell)\to\Omega_D^{\bullet+1,\bullet}(\ell),\quad
		d^{G}_{2,-1}\colon\Omega_D^{\bullet,\bullet}(\ell)\to\Omega_D^{\bullet+2,\bullet-1}(\ell).
	\end{equation*}
	For our aims, we will only need the components $d_{0,1}^G$ and $d_{1,0}^G$. They are explicitly given by the following expressions, similar to the ones found in~\cite[Chapter 4]{tondeur}: 
	%can be explicitly worked out as follows.
	%If $\eta\in\Omega^{a,b}_D(\ell)$, unravelling the definition, we can compute, for any $\square_1,\ldots,\square_{a+1}\in\Gamma(\widetilde{G})$ and $\Delta_1,\ldots,\Delta_{b+1}\in\Gamma(K)$,
	\begin{align}
		\label{eq:d_1,0}
		(d_{1,0}^G\eta)(\overline{\square}_1,&\ldots,\overline{\square}_{a+1},\Delta_1,\ldots,\Delta_b)\\%=(\mathrm{d}_D\eta)(\square_1,\ldots,\square_{a+1},\Delta_1,\ldots,\Delta_b)\\
		&=\sum_{i=1}^{a+1}(-1)^{i+1}
		\square_i(\eta(\overline{\square}_1,\ldots,\widehat{\overline{\square}}_i,\ldots,\overline{\square}_{a+1},\Delta_1,\ldots,\Delta_b))\nonumber\\
		&\hspace{0.2cm}+\sum_{1\leq i<j\leq a+1}(-1)^{i+j}\eta(\overline{[\square_i,\square_j]},\overline{\square}_1,\ldots,\widehat{\overline{\square}}_i,\ldots,\widehat{\overline{\square}}_j,\ldots,\overline{\square}_{a+1},\Delta_1,\ldots,\Delta_b)\nonumber\\
		&\hspace{0.2cm}+\sum_{i=1}^{a+1}\sum_{j=1}^b(-1)^{i+j+1}\eta(\overline{\square}_1,\ldots,\widehat{\overline{\square}}_i,\ldots,\overline{\square}_{a+1},\operatorname{pr}_K[\square_i,\Delta_j],\Delta_1,\ldots,\widehat{\Delta}_j,\ldots,\Delta_b),\nonumber
	\end{align}
	\begin{align}
		\label{eq:d_0,1}
		(d_{0,1}^G\eta)(\overline{\square}_1,&\ldots,\overline{\square}_a,\Delta_1,\ldots,\Delta_{b+1})%=(\mathrm{d}_D\eta)(\square_1,\ldots,\square_a,\Delta_1,\ldots,\Delta_{b+1})
		\\
		&=\sum_{i=1}^{b+1}(-1)^{a+i+1}\Delta_i(\eta(\overline{\square}_1,\ldots,\overline{\square}_a,\Delta_1,\dots,\widehat{\Delta}_i,\ldots,\Delta_{b+1}))\nonumber\\
		&\hspace{0.2cm}+\sum_{i=1}^{a}\sum_{j=1}^{b+1}(-1)^{i+a+j}\eta(\overline{[\square_i,\Delta_j]},\overline{\square}_1,\ldots,\widehat{\overline{\square}}_i,\ldots,\overline{\square}_a,\Delta_1,\ldots,\widehat{\Delta}_j,\ldots,\Delta_{b+1})\nonumber\\
		&\hspace{0.2cm}+\sum_{1\leq i<j\leq b+1}(-1)^{i+a+j}\eta(\overline{\square}_1,\ldots,\overline{\square}_a,[\Delta_i,\Delta_j],\Delta_1,\ldots,\widehat{\Delta}_i,\ldots,\widehat{\Delta}_j,\ldots,\Delta_{b+1}),\nonumber
	\end{align}
	where $\eta\in\Omega^{a,b}_D(\ell)$, $\square_1,\ldots,\square_{a+1}\in\Gamma(\widetilde{G})$ and $\Delta_1,\ldots,\Delta_{b+1}\in\Gamma(K)$.
	Above, we denoted by $\overline{\square}_i$ the equivalence class $\square_i\ \operatorname{mod}\ K\in\Gamma(D_\perp\ell)$.
	In view of the decomposition~\eqref{eq:decomposition_d_D}, the cohomological condition $d_D^2=0$ can be rewritten as
	\begin{numcases}{}
		d_{2,-1}^G\circ d_{2,-1}^G=0,\nonumber\\
		d_{2,-1}^G\circ d_{1,0}^G+d_{1,0}^G\circ d_{2,-1}^G=0,\nonumber\\
		d_{1,0}^G\circ d_{1,0}^G+d_{0,1}^G\circ d_{2,-1}^G+d_{2,-1}^G\circ d_{0,1}^G=0,\nonumber\\
		d_{0,1}^G\circ d_{1,0}^G+d_{1,0}^G\circ d_{0,1}^G=0,\label{eq:cohomological:4}\\
		d_{0,1}^G\circ d_{0,1}^G=0.\label{eq:cohomological:5}
	\end{numcases}
	
	The first ingredient for the construction of the desired map \eqref{eq:step2} is the following.

	%	We still need a few more ingredients for showing that, up to a sign, the map $\Phi_G$ is cochain map from $(\Omega^\bullet(\mathcal{F},\ell),\mathrm{d}_{\nabla^\ell})$ to $(\Omega^\bullet(\mathcal{F},N\mathcal{F}),\mathrm{d}_{\nabla^{N\mathcal{F}}})$.
	%They are provided by the next two lemmas.
	
	\begin{lemma}
		\label{lem:d_1,0}
		(1) For each $a\in\mathbb{N}$, we have a complex $(\Omega_D^{a,\bullet}(\ell),d_{0,1}^G)$. For $a=0$, it coincides with the de Rham complex of $K$ with coefficients in $\ell$, i.e.
		\begin{equation}\label{eq:complexes}
			\big(\Omega^{0,\bullet}_D(\ell),d^G_{0,1}\big)=\big(\Omega^\bullet(K,\ell),d_{\nabla^\ell}\big).
		\end{equation}
		
		\noindent
		(2) We have a chain map
		\[
		d_{1,0}^G\circ(\sigma|_{K})^{*}:\big(\Omega^\bullet(\mathcal{F},\ell),d_{\nabla^\ell}\big)\rightarrow\big(\Omega^{1,\bullet}_D(\ell),-d_{0,1}^G\big).
		\]
	\end{lemma}
	\begin{proof}
		For item $(1)$, note that $d_{0,1}^G$ preserves $\Omega^{a,\bullet}_D(\ell)$ and squares to zero by Eq.~\eqref{eq:cohomological:5}. Hence $(\Omega_D^{a,\bullet}(\ell),d_{0,1}^G)$ is a complex. It is clear that $\Omega_D^{0,\bullet}(\ell)=\Gamma(\wedge^\bullet K^\ast\otimes\ell)=\Omega^\bullet(K,\ell)$, and Eq.~\eqref{eq:d_0,1} shows that $d_{0,1}^G$ acts on $\Omega^{0,\bullet}_D(\ell)=\Omega^\bullet(K,\ell)$ exactly like $d_{\nabla^\ell}$.
		
		For item $(2)$, recall that under the Lie algebroid isomorphism $\sigma:K\rightarrow T\mathcal{F}$, the identity map $\text{id}_\ell\colon\ell\to\ell$ intertwines the representations of $K$ and $T\mathcal{F}$ on $\ell$. Hence, we have an isomorphism of complexes
		\[
		(\sigma|_{K})^{*}:\big(\Omega^\bullet(\mathcal{F},\ell),d_{\nabla^\ell}\big)\overset{\sim}{\longrightarrow}\big(\Omega^\bullet(K,\ell),d_{\nabla^\ell}\big).
		\]
		Moreover, Eq.~\eqref{eq:cohomological:4} shows that we have a chain map
		\begin{equation*}
			d_{1,0}^G\colon(\Omega^{0,\bullet}_D(\ell),d_{0,1}^G)\longrightarrow(\Omega^{1,\bullet}_D(\ell),-d_{0,1}^G).
		\end{equation*}
		Along with the equality of complexes \eqref{eq:complexes} obtained in part $(1)$, this proves the statement.
	\end{proof}

	%We have the following cochain morphism
	%\begin{equation*}
	%\mathrm{d}_{1,0}^G\colon\big(\Omega^{0,\bullet}_D(\ell),\mathrm{d}_{0,1}^G\big)\longrightarrow\big(\Omega^{1,\bullet}_D(\ell),-\mathrm{d}_{0,1}^G\big).
	%\end{equation*}
	
	%\begin{proof}
	%It follows directly from Equation~\eqref{eq:cohomological:4}.
	%\end{proof}
	
	%\begin{lemma}
	%For each $a\in\mathbb{N}_0$, $(\Omega_D^{a,\bullet}(\ell),\mathrm{d}_{0,1}^G)$ is a subcomplex of $(\Omega_D^\bullet(\ell),\mathrm{d}_D)$.
	%In particular, for $a=0$, it coincides with the de Rham complex of $K$ with coefficients in $\ell$, i.e.
	%\begin{equation*}
	%(\Omega^{0,\bullet}_D(\ell),\mathrm{d}^G_{0,1})=(\Omega^\bullet(K,\ell),\mathrm{d}_{\nabla^\ell}).
	%\end{equation*}
	%\end{lemma}
	
	%\begin{proof}
	%By construction $\mathrm{d}_{0,1}^G$ preserves $\Omega^{a,\bullet}_D(\ell)$, for each $a$, and it squares to zero, i.e.~$\mathrm{d}_{0,1}^G\circ\mathrm{d}_{0,1}^G=0$ (cf.~Equation~\eqref{eq:cohomological:5}).
	%Further, obviously $\Omega_D^{0,\bullet}(\ell)=\Gamma(\wedge^\bullet K^\ast\otimes\ell)=\Omega^\bullet(K,\ell)$ and, as one can easily see from Equation~\eqref{eq:d_0,1}, $\mathrm{d}_{0,1}^G$ acts on $\Omega^{0,\bullet}_D(\ell)=\Omega^\bullet(K,\ell)$ exactly like $\mathrm{d}_{\nabla^\ell}$.
	%\end{proof}

	The second ingredient is the standard degree $0$ graded $\Omega^\bullet(K)$-module isomorphism
	\begin{equation*}
		\tau\colon\Omega^{1,\bullet}_D(\ell)\overset{\sim}{\longrightarrow}\Omega^\bullet(K,J^1_\perp\ell),
	\end{equation*}
	which maps $\eta\in\Omega^{1,b}_D(\ell)$ to the element $\tau(\eta)\in\Omega^b(K,J^1_\perp\ell)$ uniquely determined by
	\begin{equation*}
		\eta(\overline{\square},\Delta_1,\ldots,\Delta_b)=\big\langle\tau(\eta)(\Delta_1,\ldots,\Delta_b),\overline{\square}\big\rangle
	\end{equation*}
	for any $\square\in\Gamma(\widetilde{G})$ and $\Delta_1,\ldots,\Delta_b\in\Gamma(K)$. This map is compatible with differentials.

	%Let us recall here the standard degree $0$ graded $\Omega^\bullet(K)$-module isomorphism
	%\begin{equation*}
	%\tau\colon\Omega^{1,\bullet}_D(\ell)\overset{\sim}{\longrightarrow}\Omega^\bullet(K,J^1_\perp\ell)
	%\end{equation*}
	%which maps $\eta\in\Omega^{1,b}_D(\ell)$ to the unique $\tau(\eta)\in\Omega^b(K;J^1_\perp\ell)$ satisfying the following identity
	%\begin{equation*}
	%\eta(\overline{\square},\Delta_1,\ldots,\Delta_b)=\langle\tau(\eta)(\Delta_1,\ldots,\Delta_b),\overline{\square}\rangle\in\Gamma(\ell)
	%\end{equation*}
	%for any $\square\in\Gamma(\widetilde{G})$ and $\Delta_1,\ldots,\Delta_b\in\Gamma(K)$.
	
	\begin{lemma}
		\label{lem:tau}
		The map $\tau$ is an isomorphism of complexes
		\begin{equation*}
			\tau\colon\big(\Omega^{1,\bullet}_D(\ell),-d_{0,1}^G\big)\longrightarrow\left(\Omega^\bullet(K,J^1_\perp\ell),d_{\nabla^{J^1_\perp\ell}}\right).
		\end{equation*}
	\end{lemma}
	
	\begin{proof}
		Fix $\eta\in\Omega^{1,b}_D(\ell)$. Using the Koszul formula for $d_{\nabla^{J^1_\perp\ell}}$ and the definition \eqref{eq:K_representation_J^1_perp ell} of $\nabla^{J^1_\perp\ell}$, we compute for any $\square\in\Gamma(\widetilde{G})$ and $\Delta_1,\ldots,\Delta_{b+1}\in\Gamma(K)$:
		\begin{align*}
			&\left\langle\left(d_{\nabla^{J^1_\perp\ell}}\tau(\eta)\right)(\Delta_1,\ldots,\Delta_{b+1}),\overline{\square}\right\rangle\\
			&\hspace{1.5cm}=\sum_{i=1}^{b+1}(-1)^{i+1}\left\langle\nabla^{J^1_\perp\ell}_{\Delta_i}\big(\tau(\eta)(\Delta_1,\ldots,\widehat{\Delta}_i,\ldots,\Delta_{b+1})\big),\overline{\square}\right\rangle\\
			&\hspace{1.5cm}\phantom{=}+\sum_{i<j}(-1)^{i+j}\left\langle\tau(\eta)\big([\Delta_i,\Delta_j],\Delta_1,\ldots,\widehat{\Delta}_i,\ldots,\widehat{\Delta}_j,\ldots,\Delta_{b+1}\big),\overline{\square}\right\rangle\\
			&\hspace{1.5cm}=\sum_{i=1}^{b+1}(-1)^{i+1}\Delta_i\big(\eta(\overline{\square},\Delta_1,\ldots,\widehat{\Delta}_i,\ldots,\Delta_{b+1})\big)\\
			&\hspace{1.5cm}\phantom{=}-\sum_{i=1}^{b+1}(-1)^{i+1}\eta\big(\overline{[\Delta_i,\square]},\Delta_1,\ldots,\widehat{\Delta}_i,\ldots,\Delta_{b+1}\big)\\
			&\hspace{1.5cm}\phantom{=}+\sum_{i<j}(-1)^{i+j}\eta\big(\overline{\square},[\Delta_i,\Delta_j],\Delta_1,\ldots,\widehat{\Delta}_i,\ldots,\widehat{\Delta}_j,\ldots,\Delta_{b+1}\big)\\
			&\hspace{1.5cm}=-\big(d_{0,1}^G\eta\big)\big(\overline{\square},\Delta_1,\ldots,\Delta_{b+1}\big)\\
			&\hspace{1.5cm}=-\left\langle\tau\big(d_{0,1}^G\eta\big)(\Delta_1,\dots,\Delta_{b+1}),\overline{\square}\right\rangle,
		\end{align*}
		where we also used the expression for $d_{0,1}^G$ in Eq.~\eqref{eq:d_0,1}. This proves the statement.
		%Using the expression for $\mathrm{d}_{0,1}^G$ in Equation~\eqref{eq:d_0,1}, the definition of $\nabla^{J^1_\perp\ell}$ in Equation~\eqref{eq:K_representation_J^1_perp ell} and the Koszul formula for $\mathrm{d}_{\nabla^{J^1_\perp\ell}}$ we can compute:
		%\begin{align*}
		%\langle(\mathrm{d}_{\nabla^{J^1_\perp\ell}}(\tau\eta))(\Delta_1,\ldots,\Delta_{b+1}),\overline{\square}\rangle&=\sum_{i=1}^{b+1}(-1)^{i+1}\langle\nabla^{J^1_\perp\ell}_{\Delta_i}((\tau\eta)(\Delta_1,\ldots,\widehat{\Delta}_i,\ldots,\Delta_{b+1})),\overline{\square}\rangle\\
		%&\phantom{=}+\sum_{i<j}(-1)^{i+j}\langle(\tau\eta)([\Delta_i,\Delta_j],\Delta_1,\ldots,\widehat{\Delta}_i,\ldots,\widehat{\Delta}_j,\ldots,\Delta_{b+1}),\overline{\square}\rangle\\
		%&=\sum_{i=1}^{b+1}(-1)^{i+1}\Delta_i(\eta(\overline{\square},\Delta_1,\ldots,\widehat{\Delta}_i,\ldots,\Delta_{b+1}))\\
		%&\phantom{=}-\sum_{i=1}^{b+1}(-1)^{i+1}\eta(\overline{[\Delta_i,\square]},\Delta_1,\ldots,\widehat{\Delta}_i,\ldots,\Delta_{b+1})\\
		%&\phantom{=}+\sum_{i<j}(-1)^{i+j}\eta(\overline{\square},[\Delta_i,\Delta_j],\Delta_1,\ldots,\widehat{\Delta}_i,\ldots,\widehat{\Delta}_j,\ldots,\Delta_{b+1})\\
		%&=-(\mathrm{d}_{0,1}^G\eta)(\overline{\square},\Delta_1,\ldots,\Delta_{b+1})\\
		%&=\langle-(\tau(\mathrm{d}_{0,1}^G\eta))(\Delta_1,\dots,\Delta_{b+1}),\overline{\square}\rangle,
		%\end{align*}
		%for any $\square\in\Gamma(\widetilde{G})$ and $\Delta_1,\ldots,\Delta_{b+1}\in\Gamma(K)$.
		%So, $\mathrm{d}_{\nabla^{J^1_\perp\ell}}\circ\tau=-\tau\circ\mathrm{d}_{0,1}^G$ on $\Omega^{1,\bullet}_D(\ell)$.
	\end{proof}
	
	Combining Lemma~\ref{lem:d_1,0} and Lemma~\ref{lem:tau}, we obtain the desired chain map \eqref{eq:step2}.
	
	\begin{cor}\label{cor:chain2}
		We have a chain map
		\[
		\tau\circ d_{1,0}^G\circ(\sigma|_{K})^{*}:\big(\Omega^{\bullet}(\mathcal{F},\ell),d_{\nabla^{\ell}}\big)\rightarrow\left(\Omega^{\bullet}(K,J^{1}_{\perp}\ell),d_{\nabla^{J^{1}_{\perp}\ell}}\right).
		\]
	\end{cor}
	
	At last, we show that $\Phi_G$ is the composition of the chain maps in Cor.~\ref{cor:chain1} and Cor.~\ref{cor:chain2}.
	
	\begin{thm}
		\label{theor:Phi_cochain_map}
		The map $\Phi_G$ 
		fits in the following commutative diagram of chain maps
		\begin{equation*}
			%\label{eq:theor:Phi_cochain_map:commutative_diagram}
			\begin{tikzcd}[row sep=large, column sep=large]
				\big(\Omega^{1,\bullet}_D(\ell),-d_{0,1}^G\big)\arrow[rrr, "\tau"]&&&\left(\Omega^\bullet(K,J^1_\perp\ell),d_{\nabla^{J^1_\perp\ell}}\right)\arrow[d, "-\wedge^\bullet(-\sigma|_K^{-1})^{*}\otimes(\sigma_\perp\circ\varpi_\perp^\sharp)"]\\
				\big(\Omega^\bullet(\mathcal{F},\ell),d_{\nabla^\ell}\big)\arrow[u, "d_{1,0}^G\circ(\sigma|_{K})^{*}"]\arrow[rrr, swap,"\Phi_G"]&&&\big(\Omega^\bullet(\mathcal{F},N\mathcal{F}),-d_{\nabla^{N\mathcal{F}}}\big)
			\end{tikzcd}.
		\end{equation*}
		In particular, it is also a chain map
		\begin{equation}
			\label{eq:theor:Phi_cochain_map}
			\Phi_G\colon\big(\Omega^\bullet(\mathcal{F},\ell),d_{\nabla^\ell}\big)\longrightarrow\big(\Omega^\bullet(\mathcal{F},N\mathcal{F}),-d_{\nabla^{N\mathcal{F}}}\big).
		\end{equation}	
	\end{thm} 
	
	\begin{proof}
		From Lemma~\ref{lem:d_1,0} and Lemma~\ref{lem:tau} we already know that the left and top arrows are chain maps. Moreover, Cor.~\ref{cor:chain1} yields a chain map
		\[
		\wedge^\bullet(\sigma|_K^{-1})^{*}\otimes(\sigma_\perp\circ\varpi_\perp^\sharp):\left(\Omega^\bullet(K,J^1_\perp\ell),d_{\nabla^{J^1_\perp\ell}}\right)\longrightarrow\left(\Omega^\bullet(\mathcal{F},N\mathcal{F}),d_{\nabla^{N\mathcal{F}}}\right),
		\]
		hence changing signs shows that the right arrow is a chain map
		\[
		-\wedge^\bullet(-\sigma|_K^{-1})^{*}\otimes(\sigma_\perp\circ\varpi_\perp^\sharp):\left(\Omega^\bullet(K,J^1_\perp\ell),d_{\nabla^{J^1_\perp\ell}}\right)\longrightarrow\left(\Omega^\bullet(\mathcal{F},N\mathcal{F}),-d_{\nabla^{N\mathcal{F}}}\right).
		\]
		It remains to check that the diagram commutes.
		Pick $\eta\in\Omega^b(\mathcal{F},\ell)$.
		For $\Delta_1,\ldots,\Delta_b\in\Gamma(K)$ and $\beta\in\Gamma(N^\ast\mathcal{F}\otimes\ell)$, we can compute
		\begin{align*}
			&\left\langle\beta,\Phi_G(\eta)\big(\sigma(\Delta_1),\ldots,\sigma(\Delta_b)\big)\right\rangle\\
			&\hspace{1cm}=d_D((\sigma\circ\operatorname{pr}_K)^\ast\eta)\left(\Delta_1,\ldots,\Delta_b,\varpi^\flat|_{\widetilde{G}}^{-1}(\sigma_\perp^\ast\beta)\right)\\
			&\hspace{1cm}=(-1)^bd_D((\sigma\circ\operatorname{pr}_K)^\ast\eta)\left(\varpi^\flat|_{\widetilde{G}}^{-1}(\sigma_\perp^\ast\beta),\Delta_1,\ldots,\Delta_b\right)\\
			&\hspace{1cm}=(-1)^b\big(d_{1,0}^G\sigma^{*}(\eta)\big)\big(\varpi_\perp^\sharp(\sigma_\perp^\ast\beta),\Delta_1,\ldots,\Delta_b\big)\\
			&\hspace{1cm}=(-1)^b\left\langle\tau\big(d_{1,0}^G\sigma^{*}(\eta)\big)(\Delta_1,\ldots,\Delta_b),\varpi_\perp^\sharp(\sigma_\perp^\ast\beta)\right\rangle\\
			&\hspace{1cm}=-(-1)^b\left\langle\sigma_\perp^\ast\beta,\varpi_\perp^\sharp\left(\tau\big(d_{1,0}^G\sigma^{*}(\eta)\big)(\Delta_1,\ldots,\Delta_b)\right)\right\rangle\\
			&\hspace{1cm}=-(-1)^b\left\langle\beta,(\sigma_\perp\circ\varpi_\perp^\sharp)\left(\tau\big(d_{1,0}^G\sigma^{*}(\eta)\big)(\Delta_1,\ldots,\Delta_b)\right)\right\rangle\\
			&\hspace{1cm}=\left\langle \beta,-\left(\wedge^\bullet(-\sigma|_K^{-1})^{*}\otimes(\sigma_\perp\circ\varpi_\perp^\sharp)\circ\tau\circ d_{1,0}^G\circ(\sigma|_{K})^{*}\right)(\eta)\big(\sigma(\Delta_1),\ldots,\sigma(\Delta_b)\big)\right\rangle.
			%&=-(-1)^b\langle\beta,((\wedge^b(\sigma|_K^\ast)^{-1}\otimes(\sigma_\perp^\sharp\circ\varpi_\perp))\circ\tau\circ\mathrm{d}_{1,0}^G)(\eta)(\sigma\Delta_1,\ldots,\sigma\Delta_b)\rangle.
			%&=\langle\beta,(-(\wedge^b(-\sigma|_K^\ast)^{-1}\otimes(\sigma_\perp^\sharp\circ\varpi_\perp))\circ\tau\circ\mathrm{d}_{1,0}^G)(\eta)(\sigma\Delta_1,\ldots,\sigma\Delta_b)\rangle.
		\end{align*}
		This shows that the diagram commutes, hence the proof is finished.
	\end{proof}

	\subsection{The induced map in cohomology}	
	\label{sec:induced_map_cohomology}
	The chain map $\Phi_G$ depends on the choice of complement $G$. This dependence is made explicit by Thm.~\ref{theor:Phi_cochain_map}, which shows that $\Phi_G$ is given by the composition
	%$\Phi_G$ introduced in Definition~\ref{def:Phi} is a cochain morphism $\Phi_G:(\Omega^\bullet(\mathcal{F},\ell),d_{\nabla^\ell})\longrightarrow(\Omega^\bullet(\mathcal{F},N\mathcal{F}),d_{\nabla^{N\mathcal{F}}})$.
	%Indeed, it gets written as the composition of cochain maps, namely
	\begin{equation*}
		\Phi_G=-\wedge^\bullet(-\sigma|_K^{-1})^{*}\otimes(\sigma_\perp\circ\varpi_\perp))\circ\tau\circ d_{1,0}^G\circ(\sigma|_{K})^{*}.
	\end{equation*}
	The maps appearing above are all canonical, except for $d_{1,0}^G$. %The latter confirms that $\Phi_G$ depends on the choice of the complement $G$ exactly like $\mathrm{d}_{1,0}^G$ does.
	In this section we show that, by contrast, the map induced in cohomology by $\Phi_G$ is canonical, i.e.~independent of the complement $G$. To prove this, it suffices to show that the map induced in cohomology by the chain map
	$$
	\tau\circ d_{1,0}^G:\big(\Omega^{\bullet}(K,\ell),d_{\nabla^{\ell}}\big)\rightarrow\left(\Omega^{\bullet}(K,J^{1}_{\perp}\ell),d_{\nabla^{J^{1}_{\perp}\ell}}\right)
	$$
	is canonical. To do so, we will first construct a canonical map $F_{\text{can}}:H^{\bullet}(K,\ell)\rightarrow H^{\bullet}(K,J^{1}_{\perp}\ell)$, and then show that $F_{\text{can}}$ agrees with the map induced in cohomology by $\tau\circ d_{1,0}^G$. The construction of $F_{\text{can}}$ presented below is the specialization to the Lie pair $(D\ell,K)$ of a general construction which works for any Lie pair. In particular, for the Lie pair $(TC,T\mathcal{F})$, the construction recovers the well-known transverse differentiation map $d_{\nu}:H^{\bullet}(\mathcal{F})\rightarrow H^{\bullet}(\mathcal{F},N^{*}\mathcal{F})$ in foliation theory (see e.g.~\cite[\S 2.3]{symplectic}).

	\bigskip
	
	The map $F_{\text{can}}\colon H^{\bullet}(K,\ell)\rightarrow H^{\bullet}(K,J^{1}_{\perp}\ell)$ will arise as the composition of two separate canonical maps. To construct the first component of $F_{\text{can}}$, recall the following complexes:
	\begin{itemize}
		\item we have the de Rham complex $(\Omega^\bullet_D(\ell),d_D)$ of $D\ell$ with values in $\ell$, whose cohomology we denote by $H^{\bullet}_{D}(\ell)$,
		\item we have the de Rham complex $(\Omega^\bullet(K,\ell),d_{\nabla^{\ell}})$ of $K$ with values in $\ell$, whose cohomology we denote by $H^\bullet(K,\ell)$,
		\item we have the subcomplex $(\Omega^\bullet_{D,K}(\ell),d_{D})\subset(\Omega^\bullet_D(\ell),d_D)$ given by 
		\begin{equation*}
			\Omega^n_{D,K}(\ell)=\{\eta\in\Omega_D^n(\ell)\, \mid\, \eta(\Delta_1,\ldots,\Delta_n)=0,\ \text{for all}\ \Delta_1,\ldots,\Delta_n\in\Gamma(K)\},
		\end{equation*}
		whose cohomology we denote by $H^\bullet_{D,K}(\ell)$.
	\end{itemize}
	%Let us start constructing what will be the first component of $F_{\text{can}}$.
	%We already have the der complex $(\Omega^\bullet_D(\ell),\mathrm{d}_D)$, i.e.~the de Rham complex of the Atiyah algebroid $D\ell$ with values in its tautological representation $\ell$, whose cohomology we denote by $\mathrm{H}_D^\bullet(\ell)$.
	%Since $K:=\ker\varpi$ is a Lie subalgebroid of $D\ell$, with inclusion map $\iota_K\colon K\to D\ell$, next to the der complex we can also introduce:
	%\begin{itemize}
	%\item the de Rham complex $(\Omega^\bullet(K,\ell),\mathrm{d}_{K,\ell})$ of $K$ with values in $\ell$, whose cohomology we denote by $\mathrm{H}^\bullet(K,\ell)$,
	%\item the cochain map $\iota_K^\ast\colon(\Omega^\bullet(\ell),\mathrm{d_D})\to(\Omega^\bullet(K,\ell),\mathrm{d}_D)$ obtained by pull-back along $\iota_K$,
	%\item the cochain subcomplex $\Omega^\bullet_{D,K}(\ell)\subset\Omega^\bullet_D(\ell)$ given by the kernel of $\iota_K^\ast$, i.e.
	%\begin{equation*}
	%\Omega^\bullet_{D,K}(\ell)=\bigoplus_n\Omega^n_{D,K}(\ell),\qquad\text{where}\ 
	%\Omega^n_{D,K}(\ell)=\{\eta\in\Omega_D^n(\ell)\, \mid\, \eta(\Delta_1,\ldots,\Delta_n)=0,\ \text{for all}\ \Delta_1,\ldots,\Delta_n\in\Gamma(K)\},
	%\end{equation*}
	%whose cohomology we denote by $\mathrm{H}^\bullet_{D,K}(\ell)$.
	%\end{itemize}
	These ingredients fit in a short exact sequence of complexes and chain maps
	\begin{equation*}
		\begin{tikzcd}
			0\arrow[r]&(\Omega^\bullet_{D,K}(\ell),d_D)\arrow[r, "\text{incl}"]&(\Omega^\bullet_D(\ell),d_D)\arrow[r, "\iota_K^\ast"]&(\Omega^\bullet(K,\ell),d_{\nabla^{\ell}})\arrow[r]&0,
		\end{tikzcd}
	\end{equation*} 
	where $\iota_K\colon K\to D\ell$ is the inclusion of the Lie subalgebroid $K\subset D\ell$.
	The corresponding long exact sequence in cohomology is
	\begin{equation*}
		\begin{tikzcd}[column sep={1em}]
			\ldots\arrow[r]&H^n_{D,K}(\ell)\arrow[r]&H^n_D(\ell)\arrow[r]&H^n(K,\ell)\arrow[r, "\mathfrak{d}"]&H^{n+1}_{D,K}(\ell)\arrow[r]&H^{n+1}_D(\ell)\arrow[r]&H^{n+1}(K,\ell)\arrow[r]&\ldots
		\end{tikzcd}
	\end{equation*}
	where the connecting map $\mathfrak{d}\colon H^n(K,\ell)\to H^{n+1}_{D,K}(\ell)$ is defined by setting
	\begin{equation*}
		%\label{eq:connecting_map}
		\mathfrak{d}[\iota_K^\ast\alpha]=[d_D\alpha].
	\end{equation*}
	%for any $\widetilde{\alpha}\in\Omega^n_D(\ell)$ such that $\mathrm{d}_{K,\ell}(\iota_K^\ast\widetilde{\alpha})=0$.
	Since the der-complex is acyclic (see Rem.~\ref{rem:contracting-homotopy}), the connecting map is an isomorphism
	\begin{equation}
		\label{eq:connecting_isomorphism}
		\mathfrak{d}\colon H^\bullet(K,\ell)\overset{\sim}{\longrightarrow}H^{\bullet+1}_{D,K}(\ell).
	\end{equation} 
	This is the first component of the map $F_{\text{can}}\colon H^{\bullet}(K,\ell)\rightarrow H^{\bullet}(K,J^{1}_{\perp}\ell)$ we aim to construct.
	
	\bigskip
	
	We now construct the second component of $F_{\text{can}}$.
	Note that for any $\alpha\in\Omega_{D,K}^{n+1}(\ell)$ there exists a unique $p(\alpha)\in\Omega^n(K,J^1_\perp\ell)$ such that
	\begin{equation}
		\label{eq:p_action}
		\alpha(\square,\Delta_1,\ldots,\Delta_n)=\left\langle p(\alpha)(\Delta_1,\ldots,\Delta_n),\overline{\square}\right\rangle,
	\end{equation}
	for $\square\in\Gamma(D\ell)$ and $\Delta_1,\ldots,\Delta_n\in\Gamma(K)$. Here we denoted $\overline{\square}:=\square\ \text{mod}\ K\in\Gamma(D_\perp\ell)$.
	%Indeed, the RHS of Equation~\eqref{eq:p_action} depends only on $\overline{\square}:=\square\ \text{mod}\ K\in\Gamma(D_\perp\ell)$ and it is skew-symmetric and $C^\infty(C)$-multilinear in $\Delta_1,\ldots,\Delta_n\in\Gamma(K)$.
	The relation \eqref{eq:p_action} defines a map
	\begin{equation*}
		%\label{eq:p_map}
		p\colon\Omega^{\bullet+1}_{D,K}(\ell)\longrightarrow\Omega^\bullet(K,J^1_\perp\ell).
	\end{equation*}
	
	\begin{lemma}
		\label{lem:p}
		We get a chain map
		\begin{equation*}
			p\colon\big(\Omega^{\bullet+1}_{D,K}(\ell),d_D\big)\longrightarrow\left(\Omega^\bullet(K,J^1_\perp\ell),-d_{\nabla^{J^1_\perp\ell}}\right).
		\end{equation*}
	\end{lemma}
	\begin{proof}
		Pick $\alpha\in\Omega^{n+1}_{D,K}(\ell)$.
		For any $\Delta_1,\ldots,\Delta_{n+1}\in\Gamma(K)$ and $\square\in\Gamma(D\ell)$, we compute
		\begin{align*}
			&\left\langle p(d_D\alpha)(\Delta_1,\ldots,\Delta_{n+1}),\overline{\square}\right\rangle\\
			&\hspace{2cm}=(d_D\alpha)(\square,\Delta_1,\ldots,\Delta_{n+1})\\
			&\hspace{2cm}=\cancel{\square(\alpha(\Delta_1,\ldots,\Delta_{n+1}))}+\sum_i(-1)^i\Delta_i\big(\alpha(\square,\Delta_1,\ldots,\widehat{\Delta}_i,\ldots,\Delta_{n+1})\big)\\
			&\hspace{2cm}\phantom{=}-\sum_i(-1)^i\alpha([\Delta_i,\square],\Delta_1,\ldots,\widehat{\Delta}_i,\ldots,\Delta_{n+1})\\
			&\hspace{2cm}\phantom{=}-\sum_{i<j}(-1)^{i+j}\alpha(\square,[\Delta_i,\Delta_j],\Delta_1,\ldots,\widehat{\Delta}_i,\ldots,\widehat{\Delta}_j,\ldots,\Delta_{n+1})\\
			%		&=\sum_i(-1)^i\Delta_i\langle(p\alpha)(\Delta_1,\ldots,\widehat{\Delta}_i,\ldots,\Delta_{n+1}),\overline{\square}\rangle\\
			%		&\phantom{=}-\sum_i(-1)^i\langle(p\alpha)(\Delta_1,\ldots,\widehat{\Delta}_i,\ldots,\Delta_{n+1}),\overline{[\Delta_i,\square]}\rangle\\
			%		&\phantom{=}-\sum_{i<j}(-1)^{i+j}\langle(p\alpha)([\Delta_i,\Delta_j],\Delta_1,\ldots,\widehat{\Delta}_i,\ldots,\widehat{\Delta}_j,\ldots,\Delta_{n+1}),\overline{\square}\rangle\\
			&\hspace{2cm}=-\sum_i(-1)^{i+1}\left\langle\nabla^{J^1_\perp\ell}_{\Delta_i}\big(p(\alpha)(\Delta_1,\ldots,\widehat{\Delta}_i,\ldots,\Delta_{n+1})\big),\overline{\square}\right\rangle\\
			&\hspace{2cm}\phantom{=}-\sum_{i<j}(-1)^{i+j}\left\langle p(\alpha)\big([\Delta_i,\Delta_j],\Delta_1,\ldots,\widehat{\Delta}_i,\ldots,\widehat{\Delta}_j,\ldots,\Delta_{n+1}\big),\overline{\square}\right\rangle\\
			&\hspace{2cm}=-\left\langle\left(d_{\nabla^{J^1_\perp\ell}}p(\alpha)\right)(\Delta_1,\ldots,\Delta_{n+1}),\overline{\square}\right\rangle.
		\end{align*}
		This shows that $p\circ d_D=-d_{\nabla^{J^1_\perp\ell}}\circ p$, so the proof is finished.
	\end{proof}

	\begin{remark}
		\label{rem:p}
		If we choose a distribution $G\subset TC$ complementary to $T\mathcal{F}$, then $\Omega^\bullet_{D,K}(\ell)$ identifies with the subcomplex
		\[
		\bigoplus_{\genfrac{}{}{0pt}{}{i\geq 1}{j\geq 0}}\Omega_D^{i,j}(\ell)\subset\big(\Omega^{\bullet}_{D}(\ell),d_{D}\big),
		\]
		where we used the bi-degree notation introduced in \eqref{eq:bigrading}. The map $p\colon\Omega_{D,K}^{\bullet+1}\to\Omega^\bullet(K,J^1_\perp\ell)$ then identifies with the following composition of chain maps
		\begin{equation*}
			\left(\bigoplus\nolimits_{\genfrac{}{}{0pt}{}{i\geq 1}{j\geq 0}}\Omega_D^{i,j}(\ell),d_D\right)\overset{\operatorname{proj}}{\longrightarrow}\big(\Omega_{D}^{1,\bullet}(\ell),d_{0,1}^{G}\big)\overset{\tau}{\longrightarrow}\left(\Omega^\bullet(K,J^1_\perp\ell),-d_{\nabla^{J^1_\perp\ell}}\right),
		\end{equation*}
		see Lemma~\ref{lem:tau}. This provides an alternative proof for Lemma~\ref{lem:p}.
	\end{remark}
	
	Since $p$ is a chain map, it induces a map in cohomology
	\begin{equation}
		\label{eq:p_cohomology}
		[p]\colon H^{\bullet+1}_{D,K}(\ell)\longrightarrow H^\bullet(K,J^1_\perp\ell),
	\end{equation}
	which is the second component of the desired canonical map $F_{\text{can}}\colon H^{\bullet}(K,\ell)\rightarrow H^{\bullet}(K,J^{1}_{\perp}\ell)$.
	
	\begin{defi}
		We define $F_{\text{can}}:=[p]\circ\mathfrak{d}$, i.e.
		\[
		F_{\text{can}}\colon H^\bullet(K,\ell)\longrightarrow H^\bullet(K,J^1_\perp\ell),\ [\iota_K^\ast\alpha]\longmapsto [p(d_D\alpha)].
		\]
	\end{defi}
	We think of $F_{\text{can}}$ as a kind of ``transverse differentiation map'' for the Lie pair $(K,D\ell)$. 
	%is defined as the composition of the maps in Equations~\eqref{eq:connecting_isomorphism} and~\eqref{eq:p_cohomology}, i.e.
	%\begin{equation*}
	%F_{\textnormal{can}}:=[p]\circ\mathfrak{d}\colon \mathrm{H}^\bullet(K,\ell)\longrightarrow\mathrm{H}^\bullet(K,J^1_\perp\ell),\ [\iota_K^\ast\widetilde{\alpha}]\longmapsto [p(\mathrm{d}_D\widetilde{\alpha})].
	%\end{equation*} 
	Now we observe that this canonical map $F_{\text{can}}$ arises from the (non-canonical) chain map given by $\tau\circ d_{1,0}^G$, for any choice of complement $G$.
	
	\begin{lemma}
		\label{lem:F_can}
		For any choice of complement $G$ to $T\mathcal{F}\subset TC$, the chain map
		\[
		\tau\circ d_{1,0}^G:\big(\Omega^{\bullet}(K,\ell),d_{\nabla^{\ell}}\big)\rightarrow\left(\Omega^{\bullet}(K,J^{1}_{\perp}\ell),d_{\nabla^{J^{1}_{\perp}\ell}}\right)
		\] 
		induces a map in cohomology which is canonical (i.e.~independent of $G$). In fact, we have 
		\begin{equation*}
			[\tau\circ d_{1,0}^G]=F_{\text{can}}\colon H^\bullet(K,\ell)\longrightarrow H^\bullet(K,J^1_\perp\ell).
		\end{equation*}
	\end{lemma}
	\begin{proof}
		Fix $\alpha\in\Omega^n(K,\ell)\simeq\Omega_D^{0,n}(\ell)$ with $d_{\nabla^{\ell}}\alpha=0$. Since $\alpha=\iota_K^{*}\text{pr}_{K}^{*}\alpha$ with $\text{pr}_{K}^{*}\alpha\in\Omega^{n}_{D}(\ell)$, the definition of $F_{\text{can}}$ shows that
		\begin{equation}\label{eq:defF}
			F_{\text{can}}[\alpha]=[p(d_{D}(\text{pr}_{K}^{*}\alpha))].
		\end{equation}
		For any $\Delta_1,\ldots,\Delta_n\in\Gamma(K)$ and $\square\in\Gamma(\widetilde{G})$, we compute
		\begin{align*}
			\left\langle p(d_{D}(\text{pr}_{K}^{*}\alpha))(\Delta_1,\ldots,\Delta_n),\overline{\square}\right\rangle&=(d_{D}(\text{pr}_{K}^{*}\alpha))(\square,\Delta_1,\ldots,\Delta_n)\\
			&=(d^G_{1,0}\alpha)(\overline{\square},\Delta_1,\ldots,\Delta_n)\\
			&=\left\langle\tau(d^G_{1,0}\alpha)(\Delta_1,\ldots,\Delta_n),\overline{\square}\right\rangle.
			%		&=\square(\widetilde{\alpha}(\Delta_1,\ldots,\Delta_n))  +\cancel{\sum_i(-1)^i\Delta_i(\widetilde{\alpha}(\square,\Delta_1,\ldots,\widehat{\Delta}_i,\ldots,\Delta_n))}\\
			%		&\phantom{=}+\sum_i(-1)^i\widetilde{\alpha}([\square,\Delta_i],\Delta_1,\ldots,\widehat{\Delta}_i,\ldots,\Delta_n)\\
			%		&\phantom{=}+\bcancel{\sum_{i<j}(-1)^{i+j}\widetilde{\alpha}([\Delta_i,\Delta_j],\square,\Delta_1,\ldots,\widehat{\Delta}_i,\ldots,\widehat{\Delta}_j,\ldots,\Delta_n)}\\
			%		&=\square(\alpha(\Delta_1,\ldots,\Delta_n))+\sum_i(-1)^i\alpha(\operatorname{pr}_K[\square,\Delta_i],\Delta_1,\ldots,\widehat{\Delta}_i,\ldots,\Delta_n)\\
			%		&\overset{\eqref{eq:d_1,0}}{=}
		\end{align*}
		This shows that $p(d_{D}(\text{pr}_{K}^{*}\alpha))=(\tau\circ d^G_{1,0})(\alpha)$. Along with \eqref{eq:defF}, this shows that
		\[
		F_{\textnormal{can}}[\alpha]=[\tau\circ d^G_{1,0}][\alpha].\qedhere
		\] %$F_{\textnormal{can}}[\alpha]=[p(\mathrm{d}_D\widetilde{\alpha})]=[(\tau\circ\mathrm{d}^G_{1,0})\alpha]=[\tau\circ\mathrm{d}^G_{1,0}][\alpha]$.
		%From here the statement follows by the arbitrariness of $\alpha$.
	\end{proof}
	
	Finally, we obtain the main result of this section.
	\begin{thm}
		\label{theor:Phi_cohomology}
		For any choice of complement $G$ to $T\mathcal{F}\subset TC$, the chain map
		\[
		\Phi_G\colon(\Omega^\bullet(\mathcal{F},\ell),d_{\nabla^\ell})\longrightarrow(\Omega^\bullet(\mathcal{F},N\mathcal{F}),-d_{\nabla^{N\mathcal{F}}})
		\]
		induces a map in cohomology which is canonical (i.e.~independent of $G$). 
	\end{thm}
	\begin{proof}
		It follows immediately from Thm.~\ref{theor:Phi_cochain_map} and Lemma~\ref{lem:F_can}.
	\end{proof}
	
	Consequently, we can denote the map in cohomology unambiguously by
	\[
	[\Phi]:H^{\bullet}(\mathcal{F},\ell)\longrightarrow H^{\bullet}(\mathcal{F},N\mathcal{F}).
	\]

	\section{Infinitesimal rigidity}\label{sec:three}
	
	In this section, we show that a compact regular coisotropic submanifold $C\subset(M,\xi)$ is infinitesimally rigid inside $\text{Def}_{\mathcal{F}}(C)$. This means that every first order deformation of $C$ in $\text{Def}_{\mathcal{F}}(C)$ is induced by a contact isotopy of $(M,\xi)$. Infinitesimal rigidity implies in particular that the deformation problem of $C$ inside $\text{Def}_{\mathcal{F}}(C)$ is unobstructed.

	\begin{prop}
		\label{prop:injective}
		The map $[\Phi]:H^{1}(\mathcal{F},\ell)\rightarrow H^{1}(\mathcal{F},N\mathcal{F})$ is injective.
	\end{prop}
	\begin{proof}
		Pick $\eta\in\Omega^1(\mathcal{F},\ell)$ such that $d_{\nabla^\ell}\eta=0$, and
		assume that $[\Phi_G][\eta]=0$ for some (hence any) complement $G$ to $T\mathcal{F}$.
		This means that there exists $Z\in\mathfrak{X}(C)$ s.t.
		\begin{equation}
			\label{eq:proof:prop:injective:hypothesis}
			\Phi_G(\eta)=d_{\nabla^{N\mathcal{F}}}(Z\ \operatorname{mod}\ T\mathcal{F}).
		\end{equation}
		
		\vspace{0.2cm}
		\noindent
		\underline{Claim:} $\eta=-d_{\nabla^\ell}(\theta_{C}(Z))$, where $\theta_{C}\in\Omega^1(C,\ell)$ is the pre-contact form on $C$.
		
		\vspace{0.1cm}
		\noindent
		Note that $\theta_{C}\in\Gamma(N^{*}\mathcal{F}\otimes\ell)$ and $\sigma_\perp^\ast\theta_C=\iota_{\mathbbm{1}}\varpi\in\Gamma(J^1_\perp\ell)$, where $\varpi$ is the pre-symplectic Atiyah form corresponding to $\theta_{C}$.
		We can compute for any $\Delta\in\Gamma(K)$, 
		\begin{align*}
			\big\langle\theta_{C},\Phi_G(\eta)\big(\sigma(\Delta)\big)\big\rangle&=\left\langle\theta_C,(\sigma_\perp\circ\varpi_{\perp}^{\sharp})\left(\tau\big(d_{1,0}^{G}\sigma^{*}\eta\big)\right)(\Delta) \right\rangle\\
			&=-\left\langle \tau\big(d_{1,0}^{G}\sigma^{*}\eta\big)(\Delta),\varpi_\perp^\sharp(\sigma_\perp^\ast\theta_C)\right\rangle\\
			&=-\left(d_{1,0}^G\sigma^{*}\eta\right)\big(\varpi_\perp^\sharp(\sigma_\perp^\ast\theta_C),\Delta\big)\\
			&=-\left(d_{1,0}^G\sigma^{*}\eta\right)\big(\overline{\mathbbm{1}},\Delta\big)\\
			&=-\mathbbm{1}(\eta(\sigma(\Delta)))-\cancel{\sigma^{*}\eta(\text{pr}_{K}[\Delta,\mathbbm{1}])},\\
		\end{align*}
		using Thm.~\ref{theor:Phi_cochain_map} in the first equality and Eq.~\ref{eq:d_1,0} in the last equality. On the other hand,
		\begin{align*}
			\left\langle\theta_C,d_{\nabla^{N\mathcal{F}}}(Z\ \text{mod}\ T\mathcal{F})\big(\sigma(\Delta)\big)\right\rangle&=\theta_C\left(\nabla^{N\mathcal{F}}_{\sigma(\Delta)}Z\ \text{mod}\ T\mathcal{F}\right)\\
			&=\theta_C\big([\sigma(\Delta),Z]\big)\\
			&=\nabla^\ell_{\sigma(\Delta)}\theta_{C}(Z)\\
			&=\big(d_{\nabla^\ell}\theta_C(Z)\big)(\sigma(\Delta)).
		\end{align*}
		In view of Eq.~\eqref{eq:proof:prop:injective:hypothesis}, this shows that $\eta=-d_{\nabla^\ell}(\theta(Z))$. This finishes the proof.
	\end{proof}
	
	%	\begin{proof}
	%		Assume that $\eta\in\Omega^{1}(\mathcal{F},l)$ is closed and $\Phi(\eta)=d_{\nabla^{N\mathcal{F}}}(Z\ \text{mod}\ T\mathcal{F})$ for $Z\in\mathfrak{X}(C)$.
	%		
	%		\vspace{0.2cm}
	%		\noindent
	%		\underline{Claim:} $\eta=d_{\nabla^{l}}\overline{Z}$, where $\overline{Z}\in\Gamma(l)$ is the class of $Z$ modulo $\xi_C$.
	%		
	%		\vspace{0.1cm}
	%		\noindent
	%		The claim can be checked locally, so let us trivialize the line bundle $l$ by means of a defining one-form $\alpha_C\in\Omega^{1}(C)$ for $\xi_C$. As before, fix a vector field $Y\in\mathfrak{X}(C)$ with $\alpha_C(Y)=1$. We can then write $\eta=\beta\otimes\overline{Y}$ for some $\beta\in\Omega^{1}(\mathcal{F})$.
	%		By Cor.~\ref{cor:equality}, the assumption that $\Phi(\beta\otimes\overline{Y})=d_{\nabla^{N\mathcal{F}}}(Z\ \text{mod}\ T\mathcal{F})$ implies in particular that
	%		\[
	%		\alpha_C([V,Z])=\beta(V),\hspace{1cm}\forall V\in\Gamma(T\mathcal{F}).
	%		\]
	%		By definition of the connection $\nabla^{l}$, see \eqref{eq:conn}, we then have for all $V\in\Gamma(T\mathcal{F})$,
	%		\[
	%		\alpha_C\left(d_{\nabla^{l}}\overline{Z}(V)\right)=\alpha_C\left(\nabla^{l}_{V}\overline{Z}\right)=\alpha_C([V,Z])=\beta(V)=\alpha_C\left(\beta(V)\overline{Y}\right).
	%		\]
	%		This implies that $d_{\nabla^{l}}\overline{Z}(V)=\beta(V)\overline{Y}$, hence indeed $d_{\nabla^{l}}\overline{Z}=\beta\otimes\overline{Y}$. This proves the claim.
	%	\end{proof}
	
	\begin{remark}
		We compare first order deformations of coisotropic submanifolds  --under the constraint that the diffeomorphism type of the characteristic foliation is fixed--  in symplectic and contact geometry. First, let $C$ be a compact coisotropic submanifold of a symplectic manifold $(M,\omega)$ with characteristic foliation $\mathcal{F}$. As shown in \cite{symplectic}, when deforming $C$ inside the class of coisotropic submanifolds whose characteristic foliation is diffeomorphic to $\mathcal{F}$, first order deformations are closed foliated one-forms $\alpha\in\Omega^{1}(\mathcal{F})$ whose cohomology class lies in the kernel of the transverse differentiation map $d_{\nu}:H^{1}(\mathcal{F})\rightarrow H^{1}(\mathcal{F},N^{*}\mathcal{F})\cong H^{1}(\mathcal{F},N\mathcal{F})$. This kernel is infinite dimensional in general, see \cite[Ex.~4.4]{symplectic}. For the same deformation problem in contact geometry, we showed that first order deformations are closed elements $\eta\in\Omega^{1}(\mathcal{F},\ell)$ whose cohomology class lies in the kernel of $[\Phi]:H^{1}(\mathcal{F},\ell)\rightarrow H^{1}(\mathcal{F},N\mathcal{F})$. By contrast, this kernel is now trivial because of Prop.~\ref{prop:injective}.
	\end{remark}

	We can now refine Def.~\ref{def:first-order-defs} as follows.
	
	\begin{cor}\label{cor:exact}
		Let $C\subset(M,\xi)$ be a compact regular coisotropic submanifold with characteristic foliation $\mathcal{F}$. When deforming $C$ inside $\text{Def}_{\mathcal{F}}(C)$, first order deformations of $C$ are $1$-coboundaries in $\big(\Omega^{\bullet}(\mathcal{F},\ell),d_{\nabla^{\ell}}\big)$.
	\end{cor}
	
	We now interpret this result in terms of the action by contact isotopies of $(M,\xi)$ on the space of coisotropic submanifolds. It is clear that this action restricts to $\text{Def}_{\mathcal{F}}(C)$. In view of Prop.~\ref{prop:equiv}, the above corollary states that a compact regular coisotropic submanifold $C$ is infinitesimally rigid with respect to the action by contact isotopies, if we deform $C$ within the class $\text{Def}_{\mathcal{F}}(C)$. This implies unobstructedness of the deformation problem.

	\begin{cor}\label{cor:unob}
		Let $C\subset(M,\xi)$ be a compact regular coisotropic submanifold with characteristic foliation $\mathcal{F}$. Then $C$ is infinitesimally rigid with respect to the action by contact isotopies, when deforming $C$ inside $\text{Def}_{\mathcal{F}}(C)$. In particular, the deformation problem of $C$ inside $\text{Def}_{\mathcal{F}}(C)$ is unobstructed.
	\end{cor}
	\begin{proof}
		Denote by $(U,\ker\gamma_G)$ the local model around $C$. Given a first order deformation $d_{\nabla^{\ell}}\lambda$ of $C$ for some $\lambda\in\Gamma(\ell)$, we construct a smooth path $\eta_t$ in $\text{Def}_{\mathcal{F}}^{U}(C)$ such that  $\dt{\eta_0}=d_{\nabla^{\ell}}\lambda$. 
		
		Recall from \S\ref{sec:one} that $\gamma_G$ is a $p^{*}\ell$-valued contact form, hence the section $-p^{*}\lambda\in\Gamma(p^{*}\ell)$ gives rise via the isomorphism \eqref{eq:contact-vf2} to a contact vector field $X_{-p^{*}\lambda}\in\mathfrak{X}(U)$ with flow $\varphi_t$. Compactness of $C$ implies that there exists $\epsilon>0$ such that the flow lines of $X_{-p^{*}\lambda}$ starting at points of $C$ exist up to time $\epsilon$. Shrinking $\epsilon$ if necessary, we may assume that $\varphi_t(C)$ is the graph of a section $\eta_t\in\Gamma(U)$ for each $t\in[0,\epsilon]$. Since $\varphi_t$ is a contactomorphism, it takes the characteristic foliation $\mathcal{F}$ of $C$ to that of the coisotropic submanifold $\text{graph}(\eta_t)$. Hence, $\eta_t$ is a smooth path in $\text{Def}_{\mathcal{F}}^{U}(C)$ as required in Def.~\ref{def:smoothpath}, by the proof of Lemma~\ref{lem:diffeos}. 
		
		\vspace{0.2cm}
		\noindent
		\underline{Claim:} $\dt{\eta_0}=d_{\nabla^{l}}\lambda$.
		
		\vspace{0.1cm}
		\noindent
		To prove the claim, we proceed as in the proof of Prop.~\ref{prop:equiv}. Setting $\psi_t:=p\circ\varphi_t|_{C}$, we obtain an isotopy of $C$ satisfying
		\begin{equation}\label{eq:iso-rel}
			\eta_t=\varphi_t\circ\psi_t^{-1}.
		\end{equation}
		Let $V_t\in\mathfrak{X}(C)$ be the time-dependent vector field of $\psi_t$. The proof of Prop.~\ref{prop:equiv} shows that
		\[
		\dt{\eta_0}=-d_{\nabla^{\ell}}(\theta_C(V_0)),
		\]
		hence the claim is proved once we show that $\lambda=-\theta_C(V_0)$. To do so, note that the equation \eqref{eq:iso-rel} implies that
		\[
		\dt{\eta_0}=\left.X_{-p^{*}\lambda}\right|_{C}-V_{0},
		\]
		which is an equality of elements in $\Gamma(TU|_{C})$. Evaluating $\gamma_G=p^{*}\theta_C+\theta_G$ on both sides of this equality gives
		\[
		0=-(p^{*}\lambda)|_{C}-\theta_C(V_0)=-\lambda-\theta_C(V_0).
		\]
		Here we use that $\theta_G$ vanishes along the zero section $C$, and that $\gamma_G(X_{-p^{*}\lambda}) =-p^{*}\lambda$ as explained in Rem.~\ref{rem:contactVF}. This shows that $\lambda=-\theta_C(V_0)$, so the claim is proved.
	\end{proof}

	\section{Rigidity}
	\label{sec:rigidity}

	In the previous section, we saw that a compact regular coisotropic submanifold $C\subset(M,\xi)$ is infinitesimally rigid in $\text{Def}_{\mathcal{F}}(C)$. This raises the question whether $C$ is actually rigid in $\text{Def}_{\mathcal{F}}(C)$. We show below that this is indeed the case. More precisely, we prove that every smooth path $C_t$ in $\text{Def}_{\mathcal{F}}(C)$ deforming $C$ is generated by a contact isotopy of $(M,\xi)$, for small enough times $t$. This extends a rigidity result due to the second author \cite{rigidity}, which concerns the special case in which the characteristic foliation $\mathcal{F}$ is given by a fibration.
	The theorem to which we devote this section is the following.
	
	\begin{thm}
		\label{thm:main}
		Let $C\subset(M,\xi)$ be a compact regular coisotropic submanifold, and let $C_t$ be a smooth path in $\text{Def}_{\mathcal{F}}(C)$ with $C_0=C$.
		Then there exists an isotopy of contactomorphisms $\psi_t\colon(M,\xi)\to(M,\xi)$, locally defined around $C$, such that $C_t=\psi_t(C)$ for small enough $t$.
	\end{thm}

	The proof of this theorem relies on two results which may be interesting in their own right.
	The first result is an extension of the Gray stability theorem for contact structures \cite[Thm.~2.2.2]{geiges} to the case of pre-contact structures (see Prop.~\ref{prop:isotopy}).
	The second one is a uniqueness statement concerning coisotropic embeddings in contact manifolds (see Prop.~\ref{prop:coisotropic_embedding}). To prove the announced extension of Gray stability, we need an auxiliary lemma.

	\color{black}
	
	\begin{lemma}
		\label{lem:compp}
		Given a line bundle $\ell\to C$, consider a smooth path of pre-contact forms $\theta_t\in\Omega^1(C,\ell)$.
		Denote by $\mathcal{F}_t$ the characteristic foliation of $(C,\theta_t)$.
		Assume that there exists an isotopy of foliated diffeomorphisms $\underline{\smash{\varphi}}_t\colon(C,\mathcal{F}_0)\longrightarrow(C,\mathcal{F}_t)$ with corresponding time-dependent vector field $V_t\in\mathfrak{X}(C)$.
		If $Z\in\Gamma(T\mathcal{F}_t)$, then
		\begin{equation*}
			\left(\frac{d}{dt}\theta_t\right)(Z)=\theta_t([V_t,Z]).
		\end{equation*}
	\end{lemma}
	\begin{proof}
		We know that $Z=(\underline{\smash\varphi}_t)_{*}(X)$ for some $X\in\Gamma(T\mathcal{F}_{0})$. Consider the family of vector fields $W_s\in\Gamma(T\mathcal{F}_{s})$ for $s\in[0,1]$ defined by
		\[
		W_s=(\underline{\smash\varphi}_s)_{*}(X).
		\]
		In particular, we have $W_t=Z$. Differentiating the equality $\theta_s(W_s)=0$ at time $s=t$ gives
		\begin{align*}
			0&=\left(\left.\frac{d}{ds}\right|_{s=t}\theta_s\right)(W_t)+\theta_t\left(\left.\frac{d}{ds}\right|_{s=t}W_s\right)=\left(\frac{d}{dt}\theta_t\right)(Z)+\theta_t\left(\left.\frac{d}{ds}\right|_{s=t}W_s\right).
		\end{align*}
		The proof is finished if we show that
		\begin{equation}\label{eq:time-dep-vec}
			\left.\frac{d}{ds}\right|_{s=t}W_s=-[V_t,Z].
		\end{equation}
		Denoting by $U_t$ the time-dependent vector field corresponding to the inverse isotopy $\underline{\smash\varphi}_t^{-1}$, we have
		\begin{equation}
			\label{eq:dds}
			\left.\frac{d}{ds}\right|_{s=t}W_s=\left.\frac{d}{ds}\right|_{s=t}(\underline{\smash\varphi}_s^{-1})^{*}(X)=(\underline{\smash\varphi}_t^{-1})^{*}\big([U_t,X]\big)=\big[(\underline{\smash\varphi}_t^{-1})^{*}(U_t),Z\big].
		\end{equation}
		It is well-known that the time-dependent vector fields $U_t$ and $V_t$ are related by
		\[
		U_t=-(\underline{\smash\varphi}_t)^{*}V_t,
		\]
		see \cite[\S 1.1]{banyaga}. Inserting this equality into \eqref{eq:dds}, we see that the equation \eqref{eq:time-dep-vec} holds.
		%This statement can be checked locally. When trivializing the line bundle $l$ by a local defining one-form $\alpha_C\in\Omega^{1}(C)$ for $\xi_C$, the local model for $(M,\xi)$ around $C$ reduces to 
		%\[
		%\big(V,\ker(p^{*}\alpha_C+\epsilon_{G}^{*}\theta_{can})\big),
		%\] 
		%where $V\subset T^{*}\mathcal{F}$, see Rem.~\ref{rem:wedge}. Hence, we can assume that $\eta_t\in\Omega^{1}(\mathcal{F})$ and that the path $\theta_t\in\Omega^{1}(C,l)$ is actually a path of one-forms
		%$
		%\alpha_t:=\alpha_C+\epsilon_{G}(\eta_t)\in\Omega^{1}(C). 
		%$
		%We now have to check that for any $Z\in\Gamma(T\mathcal{F}_t)$,
		%\begin{equation}\label{eq:check-local}
		%\frac{d}{dt}\alpha_t(Z)=\alpha_t([V_t,Z]).
		%\end{equation}
		%First note that, since $Z\in\Gamma(\ker\alpha_t)$, we have
		%\[
		%0=V_t(\alpha_t(Z))=(\pounds_{V_t}\alpha_t)(Z)+\alpha_t([V_t,Z]),
		%\]
		%so the expression \eqref{eq:check-local} is equivalent with
		%\begin{equation}\label{eq:check-local2}
		%(\pounds_{V_t}\alpha_t)(Z)+	\frac{d}{dt}\alpha_t(Z)=0.
		%\end{equation}
		%To check that this equality holds, recall that $Z=d\underline{\varphi}_t(X)$ for some $X\in\Gamma(T\mathcal{F})$ since $\underline{\varphi}_t:(C,\mathcal{F})\rightarrow (C,\mathcal{F}_{t})$ is a foliated diffeomorphism. If we differentiate
		%$
		%(\underline{\varphi}_t^{*}\alpha_t)(X)=0,
		%$
		%then we get
		%\[
		%\underline{\varphi}_t^{*}\left(\pounds_{V_t}\alpha_t+\frac{d}{dt}\alpha_t\right)(X)=0,
		%\]
		%which implies that the equality \eqref{eq:check-local2} holds.
	\end{proof}

	We can now state and prove the anticipated generalization of the Gray stability theorem.
	In what follows, given a line bundle automorphism $\Phi\colon\ell\to\ell$, we use the pushforward of derivations $D\Phi\colon D\ell\rightarrow D\ell$ and the pullback of Atiyah forms $\Phi^\ast\colon\Omega^\bullet_D(\ell)\rightarrow\Omega^\bullet_D(\ell)$, which are recalled in the Appendix.

	\begin{prop}
		\label{prop:isotopy}
		Given a line bundle $\ell\to C$, consider a smooth path of pre-contact forms $\theta_t\in\Omega^1(C,\ell)$ and let $\varpi_t\in\Omega^2_D(\ell)$ be the corresponding smooth path of pre-symplectic Atiyah forms.
		Denote by $\mathcal{F}_t$ the characteristic foliation of $(C,\theta_t)$ and by $K_t$ the kernel of $\varpi_t$, so that $T\mathcal{F}_t=\sigma(K_t)$.
		Then the following are equivalent:
		\begin{enumerate}[label=(\arabic*)]
			\item\label{enumitem:prop:isotopy:1}
			there exists an isotopy of foliated diffeomorphisms $\underline{\smash{\varphi}}_t\colon(C,\mathcal{F}_0)\longrightarrow(C,\mathcal{F}_t)$,
			\item\label{enumitem:prop:isotopy:2}
			there exists an isotopy of LB-automorphisms $\Phi_t\colon\ell\to\ell$ such that $\Phi_t^\ast\varpi_t=\varpi_0$,
			\item\label{enumitem:prop:isotopy:3}
			there exists an isotopy of LB-automorphisms $\Phi_t\colon\ell\to\ell$ such that $\Phi_t^\ast\theta_t=\theta_0$.
		\end{enumerate}
	\end{prop}
	\begin{proof}
		\ref{enumitem:prop:isotopy:2}$\Longleftrightarrow$\ref{enumitem:prop:isotopy:3}
		This is easily proved using the relation $\varpi_t=d_D(\sigma^\ast\theta_t)$ or equivalently, the relation $\iota_{\mathbbm{1}}\varpi_t=\sigma^\ast\theta_t$, together with the natural identities $\Phi_t^\ast\iota_{\mathbbm{1}}=\iota_{\mathbbm{1}}\Phi_t^\ast$ and $\Phi_t^\ast d_{D}=d_{D}\Phi_t^\ast$.
		
		\ref{enumitem:prop:isotopy:2}$\Longrightarrow$\ref{enumitem:prop:isotopy:1}
		Denote by $\underline{\smash{\varphi}}_t\colon C\to C$ the isotopy of diffeomorphisms underlying $\Phi_t\colon\ell\to\ell$.
		Then the hypothesis $\varpi_0=\Phi_t^\ast\varpi_t$ can be written explicitly as follows:
		\begin{equation*}
			(\varpi_{0})_{x}(\delta_x,\delta_x')=(\Phi_{t})_{x}^{-1}\left((\varpi_{t})_{\underline{\smash{\varphi}}_t(x)}\left((D_x\Phi_t)\delta_x,(D_x\Phi_t)\delta_x'\right)\right),
		\end{equation*}
		for all $x\in C$ and $\delta_x,\delta_x'\in D_x\ell$.
		This implies that the isomorphism $D_x\Phi_t\colon D_x\ell\overset{\sim}{\longrightarrow} D_{\underline{\smash{\varphi}}_t(x)}\ell$ induces an isomorphism
		\begin{equation*}
			D_x\Phi_t:(K_{0})_{x}\overset{\sim}{\longrightarrow}(K_{t})_{\underline{\smash{\varphi}}_t(x)}.
		\end{equation*}
		Since $\sigma\colon D\ell\to TC$ induces a Lie algebroid isomorphism $K_t\overset{\sim}{\rightarrow}T\mathcal{F}_t$ and $\sigma\circ(D\Phi_t)=(\underline{\smash{\varphi}}_t)_{*}\circ\sigma$, it follows that $(\underline{\smash{\varphi}}_t)_{*}$ descends to an isomorphism
		\begin{equation*}
			(\underline{\smash{\varphi}}_t)_{*}:T_x\mathcal{F}_0\overset{\sim}{\longrightarrow}T_{\underline{\smash{\varphi}}_t(x)}\mathcal{F}_t.
		\end{equation*}
		
		\ref{enumitem:prop:isotopy:1}$\Longrightarrow$\ref{enumitem:prop:isotopy:2}
		Following the usual Moser argument, we will obtain the desidered isotopy of LB automorphisms as the flow of a suitable smooth path in $\Gamma(D\ell)$.
		For any isotopy of LB automorphisms $\Phi_t\colon\ell\to\ell$ generated by a smooth path $\square_t\in\Gamma(D\ell)$, we can compute
		\begin{equation*}
			\frac{d}{dt}\Phi_t^\ast\varpi_t=\Phi_t^\ast\left(\mathcal{L}_{\square_t}\varpi_t+\dt{\varpi}_t\right)=\Phi_t^\ast d_D\left(\iota_{\square_t}\varpi_t+\sigma^\ast\dt{\theta}_t\right),
		\end{equation*}
		where we also used Prop.~\ref{prop:relation-luca}. 
		Since the der-complex is acyclic (see Rem.~\ref{rem:contracting-homotopy}), this shows that there is an isotopy $\Phi_t$ satisfying $\Phi_t^\ast\varpi_t=\varpi_0$ if and only if we can find smooth paths $\square_t\in\Gamma(D\ell)$ and $\lambda_t\in\Gamma(\ell)$ such that
		\begin{equation}\label{eq:find2}
			\iota_{\square_t}\varpi_t=j^1\lambda_t-\sigma^\ast\dt\theta_t.
		\end{equation}
		Here we used that $d_{D}\lambda=j^1\lambda$ for any $\lambda\in\Gamma(\ell)$ under the identification $(D\ell)^\ast\otimes\ell\cong J^1\ell$.
		
		Denote by $V_t\in\mathfrak{X}(C)$ the time-dependent vector field generating the isotopy $\underline{\smash{\varphi}}_t$.
		Setting $\lambda_t=-\theta_t(V_t)$, we claim that the equation \eqref{eq:find2} can be solved for $\square_t$. To prove this, we have to check that
		$
		j^1\lambda_t-\sigma^\ast\dt\theta_t\in\Gamma(J^1\ell)
		$
		belongs to the image of $\varpi_t^\flat\colon D\ell\rightarrow J^1\ell$, which coincides with the annihilator of $K_t:=\ker\varpi_t$.
		%	To see why, note that this image coincides with the elements of $(Dl)^{*}\otimes l$ that annihilate $K_t=\ker\overline{\omega_t}$.
		%	Hence, by Lemma~\ref{lem:Kt}, the claim follows if we show that derivations of the form ${}^{t}\nabla^{l}_{Z}$ for $Z\in\Gamma(T\mathcal{F}_t)$ are annihilated by $\frac{d}{dt}\theta_t+j^{1}\big(\theta_t(V_t)\big)$.
		For any $\Delta\in\Gamma(K_t)$, we can compute
		\begin{align*}
			\left\langle\sigma^\ast\dt\theta_t+j^{1}\big(\theta_t(V_t)\big),\Delta\right\rangle&=\dt\theta_t(\sigma(\Delta))+\Delta\big(\theta_t(V_t)\big)\\
			&=\dt\theta_t(\sigma(\Delta))+\nabla^{\ell}_{\sigma(\triangle)}\theta_t(V_t)\\
			&=\dt\theta_t(\sigma(\Delta))+\theta_t\big([\sigma(\triangle),V_t]\big)\\
			&=\theta_t\big([V_t,\sigma(\triangle)]\big)+\theta_t\big([\sigma(\triangle),V_t]\big)\\
			&=0,
		\end{align*}
		%	\begin{align*}
		%		\Big(\frac{d}{dt}\theta_t+j^{1}\big(\theta_t(V_t)\big)\Big)({}^{t}\nabla^{l}_{Z})&=\frac{d}{dt}\theta_t(\sigma({}^{t}\nabla^{l}_{Z}))+{}^{t}\nabla^{l}_{Z}\big(\theta_t(V_t)\big)\\
		%		&=\frac{d}{dt}\theta_t(Z)+\theta_t([Z,V_t])\\
		%		&=0,
		%	\end{align*}
		where we used the diagram \eqref{eq:nabla^ell} in the second equality and Lemma \ref{lem:compp} in the fourth equality.
		
		This shows that setting $\lambda_t:=-\theta_t(V_t)$, the equation \eqref{eq:find2} can be solved for $\square_t$.
		Moreover, the kernels of $\varpi_t$ for $t\in[0,1]$ form a vector bundle over $C\times[0,1]$, and one can choose a complementary subbundle inside the pullback bundle $\operatorname{pr}_1^\ast D\ell$, where $\operatorname{pr}_1\colon C\times[0,1]\rightarrow C$ denotes the projection.
		Requiring that $\square_t$ takes values in this subbundle uniquely determines $\square_t$, and shows in particular that $\square_t$ can be chosen in a smooth manner.
		Since $C$ is compact, we can integrate $\square_t$ globally to an isotopy of LB automorphisms $\Phi_t\colon\ell\to\ell$ satisfying $\Phi_t^\ast\varpi_t=\varpi_0$ for all $t$. This finishes the proof.
	\end{proof}
	
	\begin{remark}
		A weaker version of the implication \ref{enumitem:prop:isotopy:1}$\Longrightarrow$\ref{enumitem:prop:isotopy:2} in the above proposition has appeared in~\cite[Section 4.2]{rigidity}.
		There one claims (without proof) that, if the foliation $\mathcal{F}_0$ is simple, one can find an isotopy of LB automorphisms $\Phi_t\colon\ell\to\ell$ satisfying $D\Phi_t(K_0)=K_t$ for any $t$.
		We proved a stronger statement, namely that the isotopy $\Phi_t$ can be chosen so that $\Phi_t^\ast\varpi_t=\varpi_0$ for any $t$, without imposing any restrictions on the foliation $\mathcal{F}_0$.
	\end{remark}

	As anticipated, the second ingredient for the proof of our main result is the following uniqueness statement concerning coisotropic embeddings in contact manifolds.
	
	\begin{prop}
		\label{prop:coisotropic_embedding}
		Let $L\to M$ be a line bundle,  $\Omega_t\in\Omega_D^2(L)$ a smooth path of symplectic Atiyah forms for $0\leq t\leq 1$, and $C\subseteq M$ a submanifold.
		Set $\ell:=L|_C$ and let $I\colon\ell\to L$ be the inclusion.
		Assume that $C$ is a compact regular coisotropic submanifold of $(M,\Omega_t)$ for all $0\leq t\leq 1$, such that additionally
		\begin{equation}
			\label{eq:prop:coisotropic_embedding:1}
			I^\ast\Omega_t=I^\ast\Omega_0.
		\end{equation}
		%So, in particular, $K_t:=\ker\varpi_t$ and $T\mathcal{F}_t:=\sigma(K_t)$ do not depend on $t$.
		Then there exist a neighborhood $\mathcal{U}\subset M$ of $C$ and a path of LB morphisms $\Psi_t\colon L|_\mathcal{U}\to L$ covering a path of embeddings $\psi_t\colon\mathcal{U}\to M$, such that
		\begin{equation}\label{eq:prop:coisotropic_embedding:2}
			\Psi_0=\operatorname{id}_{\mathcal{U}},\quad\text{and}\quad	
			\Psi_t^\ast\Omega_t=\Omega_0|_{\mathcal{U}}.
		\end{equation}
		Moreover, $\Psi_t$ induces an isotopy of LB automorphisms $\Psi_t|_C\colon\ell\to\ell$ covering the isotopy of diffeomorphisms $\psi_t|_C\colon C\to C$, such that 
		$
		\Psi_t|_C^\ast(I^\ast\Omega_0)=I^\ast\Omega_0.
		$
	\end{prop}
	\begin{proof}
		Set $\square_t=-\Omega_t^\sharp(\iota_{\mathbbm{1}}\dt\Omega_t)\in\Gamma(DL)$. 
		We claim that $\square_t|_{C}\in\Gamma(D\ell)$, or in other words that $\sigma(\square_t)$ is tangent to $C$. To see this, note that by \eqref{eq:prop:coisotropic_embedding:1}, we have $I^\ast\dt\Omega_t=0$. Since $\mathbbm{1}|_C\in\Gamma(D\ell)$, this implies that for all $\Delta\in\Gamma(D\ell)$,
		\begin{equation*}
			\big.\big(\iota_{\mathbbm{1}}\dt\Omega_t\big)\big|_{C}(\Delta)=(I^\ast\dt\Omega_t)(\mathbbm{1}|_{C},\Delta)=0.
		\end{equation*}
		This shows that
		\begin{equation}\label{eq:along}
			\big.\big(\iota_{\mathbbm{1}}\dt\Omega_t\big)\big|_{C}\in\Gamma((D\ell)^\circ).
		\end{equation}
		Since $C\subseteq (M,\Omega_t)$ is a regular coisotropic submanifold for all $0\leq t\leq 1$, we know by Prop.~\ref{prop:coiso-atiyah} that $(D\ell)^{\perp_{\Omega_t}}\subset D\ell$. Here $(D\ell)^{\perp_{\Omega_t}}$ denotes the $\Omega_t$-orthogonal of $D\ell$ inside $DL$.
		Along with \eqref{eq:along}, this implies that
		\begin{equation*}
			\square_t|_C=\big.\Omega_t^\sharp(-\iota_{\mathbbm{1}}\dt\Omega_t)\big|_{C}\in\Gamma\big(\Omega_t^\sharp((D\ell)^\circ)\big)=\Gamma\big((D\ell)^{\perp_{\Omega_t}}\big)\subset\Gamma(D\ell).
		\end{equation*}
		This confirms that $\sigma(\square_t)$ is tangent to $C$. Since $C$ is compact, there exists a neighborhood $\mathcal{U}\subset M$ of $C$ on which the flow $\psi_t$ of $\sigma(\square_t)$ is defined for all times $0\leq t\leq 1$. Hence, the isotopy of line bundle automorphisms $\Psi_t\colon L\to L$ integrating $\square_t\in\Gamma(DL)$ is defined on $L|_{\mathcal{U}}$.

		The fact that $\Psi_t$ satisfies the equation \eqref{eq:prop:coisotropic_embedding:2} follows from the computation
		\begin{equation*}
			\frac{d}{dt}\Psi_t^\ast\Omega_t=\Psi_t^\ast\left(\mathcal{L}_{\square_t}\Omega_t+\dt{\Omega}_t\right)=\Psi_t^\ast d_D\left(\iota_{\square_t}\Omega_t+\iota_{\mathbbm{1}}\dt{\Omega}_t\right)=0.
		\end{equation*}
		Here we used that $\iota_{\mathbbm{1}}$ is a contracting homotopy for the der-complex, see Rem.~\ref{rem:contracting-homotopy}. At last, since $\sigma(\square_t)$ is tangent to $C$, it is clear that $\Psi_t$ induces an isotopy of LB automorphisms $\Psi_t|_C\colon\ell\to\ell$ covering $\psi_t|_C\colon C\to C$. This isotopy automatically satisfies $\Psi_t|_C^\ast(I^\ast\Omega_0)=I^\ast\Omega_0$, since $\Psi_t^\ast\Omega_t=\Omega_0|_{\mathcal{U}}$ and $I^\ast\Omega_t=I^\ast\Omega_0$. This finishes the proof.
	\end{proof}

	We can now prove Thm.~\ref{thm:main}. In the following, $C\subset(M,\xi)$ is a compact regular coisotropic submanifold with pre-contact form $\theta_C\in\Omega^{1}(C,\ell)$, pre-symplectic Atiyah form $\varpi\in\Omega^{2}_{D}(\ell)$ and characteristic foliation $\mathcal{F}$. To prove Thm.~\ref{thm:main}, we will work in the local model $(U,\ker\gamma_G)$ for $(M,\xi)$ around $C$. Recall from \S\ref{sec:one} that the construction of this model involves a choice of complement $G$ to $T\mathcal{F}$.  The contact structure $\ker\gamma_G$ lives on a neighborhood $U$ of the zero section in the bundle $p:T^{*}\mathcal{F}\otimes\ell\rightarrow C$, and $\gamma_G$ is a $(p^{*}\ell)$-valued contact form. Working in $(U,\ker\gamma_G)$ instead of $(M,\xi)$, the setup is as follows.
	\begin{itemize}
		\item The smooth path $C_t$ in $\text{Def}_{\mathcal{F}}(C)$ is identified with a smooth path of coisotropic sections $\eta_t\in\text{Def}^{U}_{\mathcal{F}}(C)$ with $\eta_0=0$.
		\item By Lemma~\ref{lem:coiso-in-model}, the path $\eta_t$ gives rise to a path of pre-contact forms
		\begin{equation*}
			\theta_t:=\theta_C+\text{pr}_{T\mathcal{F}}^{*}\eta_t\in\Omega^{1}(C,\ell).
		\end{equation*}
		Denote by $\mathcal{F}_t$ the characteristic foliation of the pre-contact structure $\ker\theta_t$. 
		\item By Def.~\ref{def:smoothpath}, we have an isotopy $\underline{\varphi}_t$ of foliated diffeomorphisms
		\[
		\underline{\varphi}_t:(C,\mathcal{F})\rightarrow (C,\mathcal{F}_{t}).
		\]
		\item Denote by $\varpi_t\in\Omega^{2}_{D}(\ell)$ the pre-symplectic Atiyah form corresponding to $\theta_t$.
	\end{itemize}
	\noindent
	Every $\eta\in\Gamma(T^\ast\mathcal{F}\otimes\ell)$ determines a fiber-preserving diffeomorphism
	\begin{equation*}
		\tau_{\eta}\colon T^\ast\mathcal{F}\otimes\ell\overset{\sim}{\longrightarrow}T^\ast\mathcal{F}\otimes\ell,\ \alpha_x\longmapsto\alpha_x+\eta_x,
	\end{equation*}
	and a LB-automorphism $T_\eta\colon p^\ast\ell\longrightarrow p^\ast\ell$ covering $\tau_\eta$, given by 
	\begin{equation*}
		T_\eta\colon p^\ast\ell\overset{\sim}{\longrightarrow}p^\ast\ell,\ (\alpha_x,\lambda_x)\longmapsto(\alpha_x+\eta_x,\lambda_x).
	\end{equation*}
	By the proof of Lemma~\ref{lem:coiso-in-model}, the pre-symplectic Atiyah forms $\varpi_t\in\Omega^{2}_{D}(\ell)$ are given by
	\begin{equation}\label{eq:varpi-I}
		\varpi_t=I^{*}\left(T_{\eta_t}^{*}\Omega_G\right),
	\end{equation}
	where $\Omega_G\in\Omega^{2}_{D}(p^{*}\ell)$ is the symplectic Atiyah form corresponding to the contact form $\gamma_G$ and $I:\ell\hookrightarrow p^{*}\ell$ is the inclusion map. We now proceed with the proof of Thm.~\ref{thm:main}.

	\begin{proof}[Proof of Thm.~\ref{thm:main}]
		By the implication \ref{enumitem:prop:isotopy:1}$\Longrightarrow$\ref{enumitem:prop:isotopy:2} in Prop.~\ref{prop:isotopy}, there exists an isotopy of LB automorphisms $\Phi_t\colon\ell\overset{\sim}{\to}\ell$ covering an isotopy of diffeomorphisms $\varphi_t\colon C\overset{\sim}{\to}C$, such that
		\begin{equation}
			\label{eq:proof:thm:main:1}
			\Phi_t^\ast\varpi_t=\varpi_0,\ \text{for all}\ t.
		\end{equation}
		Since $C$ is compact, we can assume that $(d\varphi_t)^{-1}(T\mathcal{F})$ is transverse to $G$, i.e.
		\[
		TC=G\oplus(d\varphi_t)^{-1}(T\mathcal{F}),
		\]
		by restricting to small enough times $t$.
		Hence, one can build an isotopy of VB automorphisms $\widetilde\varphi_t\colon T^\ast\mathcal{F}\otimes\ell\overset{\sim}{\longrightarrow}T^\ast\mathcal{F}\otimes\ell$ covering the isotopy of diffeomorphisms $\varphi_t\colon C\overset{\sim}{\longrightarrow} C$, by setting
		\begin{equation*}
			\widetilde\varphi_t\colon T^\ast\mathcal{F}\otimes\ell\overset{\sim}{\longrightarrow}T^\ast\mathcal{F}\otimes\ell,\ \alpha_x\longmapsto\alpha_x\circ\operatorname{pr}_{T_x\mathcal{F}}\circ(d_x\varphi_t)^{-1}|_{T_{\varphi_t(x)}\mathcal{F}}.
		\end{equation*}
		Further, one can also build an isotopy of LB automorphisms $\widetilde\Phi_t\colon p^\ast\ell\overset{\sim}{\longrightarrow}p^\ast\ell$, covering the isotopy of VB automorphisms $\widetilde\varphi_t\colon T^\ast\mathcal{F}\otimes\ell\overset{\sim}{\longrightarrow}T^\ast\mathcal{F}\otimes\ell$, as follows
		\begin{equation*}
			%\label{eq:proof:thm:main:2}
			\widetilde\Phi_t\colon p^\ast\ell\overset{\sim}{\longrightarrow}p^\ast\ell,\ (\alpha_x,\lambda_x)\longmapsto(\widetilde\varphi_t(\alpha_x),\Phi_t(\lambda_x)).
		\end{equation*}
		By construction, the following diagram of VB morphisms commutes
		\begin{equation}
			\label{eq:proof:thm:main:3}
			\begin{tikzcd}
				\ell\arrow[rr, "I"]\arrow[d, swap, "\Phi_t"]&&p^\ast\ell\arrow[d, "\widetilde\Phi_t"]\\
				\ell\arrow[rr, swap, "I"]&&p^\ast\ell
			\end{tikzcd}.
		\end{equation}
		Next, setting $\Psi_t:=T_{\eta_t}\circ\widetilde{\Phi}_t$ and $\psi_t:=\tau_{\eta_t}\circ\widetilde{\varphi}_t$, we get an isotopy of LB automorphisms $\Psi_t\colon p^\ast\ell\overset{\sim}{\to} p^\ast\ell$ covering the isotopy of diffeomorphisms $\psi_t\colon T^\ast\mathcal{F}\otimes\ell\overset{\sim}{\longrightarrow}T^\ast\mathcal{F}\otimes\ell$, such that the following diagram of VB automorphisms commutes
		\begin{equation}
			\label{eq:proof:thm:main:4}
			\begin{tikzcd}
				p^\ast\ell\arrow[rr, equal]\arrow[d, swap, "\widetilde\Phi_t"]&&p^\ast\ell\arrow[d, "\Psi_t"]\\
				p^\ast\ell\arrow[rr, swap, "T_{\eta_t}"]&&p^\ast\ell
			\end{tikzcd}.
		\end{equation}
		\noindent
		We now claim that
		\begin{equation}\label{eq:ddt}
			\frac{d}{dt}\left(I^\ast(\Psi_t^\ast\Omega_G)\right)=0.
		\end{equation}
		To see why, note that by the commutative diagrams~\eqref{eq:proof:thm:main:3} and~\eqref{eq:proof:thm:main:4}, we have
		\[
		I^\ast(\Psi_t^\ast\Omega_G)=\Phi_t^{*}\left(I^{*}\left(T_{\eta_t}^{*}\Omega_G\right)\right)=\Phi_t^{*}\varpi_t,
		\]
		using \eqref{eq:varpi-I} in the last equality. Hence, the equation \eqref{eq:ddt} immediately follows from \eqref{eq:proof:thm:main:1}.
		%The equation ~\eqref{eq:proof:thm:main:1} and the commutative diagrams~\eqref{eq:proof:thm:main:3} and~\eqref{eq:proof:thm:main:4} imply that
		%\begin{equation*}
		%\frac{\mathrm{d}}{\mathrm{d}t}\left(I^\ast(\Psi_t^\ast\varpi_G)\right)=0,
		%\end{equation*}
		%where we have used that $\varpi_t=I^\ast T_{\eta_t}^\ast\varpi_G$, for all $t$.
		
		We are now in a position where we can apply Prop.~\ref{prop:coisotropic_embedding} to the smooth path of symplectic Atiyah forms $\Psi_t^\ast\Omega_G\in\Omega_D^2(p^\ast\ell)$. Note indeed that $C$ is regular coisotropic with respect to $\Psi_t^\ast\Omega_G$, because $\text{graph}(\eta_t)$ is regular coisotropic with respect to $\Omega_G$ and the diffeomorphism $\psi_t$ underlying $\Psi_t$ takes $C$ to $\text{graph}(\eta_t)$. Hence,
		for a suitable neighborhood $\mathcal{U}$ of $C$ in $T^\ast\mathcal{F}\otimes\ell$ and small enough $t$, there exists a path of LB morphisms $\widetilde\Psi_t\colon (p^\ast\ell)|_\mathcal{U}\to p^\ast\ell$, covering a path of embeddings $\widetilde\psi_t\colon\mathcal{U}\to T^\ast\mathcal{F}\otimes\ell$, such that
		\begin{equation}
			\label{eq:proof:thm:main:5}
			\widetilde\Psi_0=\operatorname{id}_{\mathcal{U}},\quad\widetilde{\psi}_t(C)=C\quad\text{and}\quad\widetilde\Psi_t^\ast\left(\Psi_t^\ast\Omega_G\right)=\Omega_G|_{\mathcal{U}}.
		\end{equation}
		%that we can apply the Coisotropic Embedding Theorem~\ref{prop:coisotropic_embedding} to the smooth path $\Psi_t^\ast\varpi_G\in\Omega_D^2(p^\ast\ell)$ of regular presymplectic Atiyah forms on $T^\ast\mathcal{F}\otimes\ell$.
		%So,
		%for a suitable neighborhood $\mathcal{U}$ of $C$ in $T^\ast\mathcal{F}\otimes\ell$, and small enough $t$, there exists a path of LB morphisms $\widetilde\Psi_t\colon (p^\ast\ell)|_\mathcal{U}\to p^\ast\ell$, covering a path of embeddings $\widetilde\psi_t\colon\mathcal{U}\to T^\ast\mathcal{F}\otimes\ell$, such that
		%\begin{equation}
		%\label{eq:proof:thm:main:5}
		%\widetilde\Psi_0=\operatorname{id}_{\mathcal{U}},\quad\widetilde{\psi}_t(C)=C,\quad\text{and}\quad\widetilde\Psi_t^\ast\Psi_t^\ast\varpi_G=\varpi_G|_{\mathcal{U}},\quad\text{for all $t$}.
		%\end{equation}
		This way, we obtain an isotopy of contactomorphisms $\psi_t\circ\widetilde\psi_t$ of the local model $(U,\ker\gamma_G)$, which are defined locally around $C$ and satisfy $\text{graph}(\eta_t)=(\psi_t\circ\widetilde\psi_t)(C)$ for all $t$. 
	\end{proof}

	\color{black}

	\section{Application to the classical coisotropic deformation problem}\label{sec:five}
	
	Thm.~\ref{thm:main} is a rigidity result for a restricted version of the coisotropic deformation problem, which only allows deformations inside the class $\text{Def}_{\mathcal{F}}(C)$. Combining this result with a well-known rigidity theorem for foliations, we obtain a rigidity result for the classical coisotropic deformation problem, which allows \emph{all} coisotropic deformations of $C$.
	
	The rigidity result in foliation theory we are referring to is due to Hamilton \cite{hamilton} and Epstein-Rosenberg \cite{epstein}. It concerns compact manifolds $C$ endowed with a Hausdorff foliation $\mathcal{F}$, meaning that the leaf space $C/\mathcal{F}$ is Hausdorff when endowed with the quotient topology. If $C$ is connected, then there exists a generic leaf $L$ such that all leaves in a saturated dense open subset are diffeomorphic to $L$. Hamilton and Epstein-Rosenberg showed that, if the generic leaf satisfies $H^{1}(L)=0$, then the foliation $\mathcal{F}$ is rigid, i.e. any foliation $\mathcal{F'}$ on $C$ close enough to $\mathcal{F}$ is conjugate to $\mathcal{F}$ by a diffeomorphism. We will use a smooth $1$-parameter version of this rigidity statement due to Del Hoyo-Fernandes \cite{rui}, which concerns smooth paths of foliations $\mathcal{F}_{t}$ deforming $\mathcal{F}$ rather than foliations $\mathcal{F'}$ close to $\mathcal{F}$.
	
	\begin{cor}\label{cor:hamilton}
		Let $C\subset(M,\xi)$ be a compact, connected regular coisotropic submanifold with characteristic foliation $\mathcal{F}$. Assume that $\mathcal{F}$ is Hausdorff and that its generic leaf $L$ satisfies $H^{1}(L)=0$. Then $C$ is rigid, i.e. if $C_t$ is a smooth path of coisotropic submanifolds with $C_0=C$, then there is an isotopy $\psi_t$ of locally defined contactomorphisms such that $C_t=\psi_t(C)$ for small enough $t$.
	\end{cor}
	\begin{proof}
		By Thm.~\ref{thm:main}, we only have to show that $C_t$ is a smooth path in $\text{Def}_{\mathcal{F}}(C)$ for small enough $t$. Working in the local model $(U,\ker\gamma_G)$ for $(M,\xi)$ around $C$, the smooth path $C_t$ is identified with a smooth path of regular coisotropic submanifolds $\text{graph}(\eta_t)$, for sections $\eta_t\in\Gamma(U)$ with $\eta_0=0$. By Lemma~\ref{lem:diffeos}, we only need to check whether there exists a smooth path of foliated diffeomorphisms
		\[
		\varphi_t:(C,\mathcal{F})\overset{\sim}{\rightarrow}\left(C,\mathcal{F}_{\eta_t}\right),
		\]
		where $\mathcal{F}_{\eta_t}$ is the characteristic foliation of the pre-contact structure $\ker\big(\theta_C+\operatorname{pr}_{T\mathcal{F}}^{*}\eta\big)$. The assumptions on $\mathcal{F}$ ensure that such diffeomorphisms $\varphi_t$ exist for small enough $t$, by \cite{rui}.
	\end{proof}
	
	\begin{ex}
		Consider the manifold $C:=S^{2}\times\mathbb{T}^{3}$ with hyperplane distribution $\xi_C$ given at points $(p,q)\in S^{2}\times\mathbb{T}^{3}$ by
		\[
		(\xi_C)_{(p,q)}=T_{p}S^{2}\oplus\mathbb{R}\left(\partial_{\theta_1}\right)_{q}\oplus\mathbb{R}\left(\cos\theta_1\partial_{\theta_2}-\sin\theta_1\partial_{\theta_3}\right)_{q}.
		\]
		Note that $\xi_C$ is the kernel of the nowhere vanishing one-form
		\[
		\theta_C=\sin\theta_1 d\theta_2+\cos\theta_1 d\theta_3.
		\]
		We claim that $\xi_C$ is a pre-contact distribution with characteristic foliation given by the fibers of the bundle $S^{2}\times\mathbb{T}^{3}\rightarrow\mathbb{T}^{3}$. To see this, we trivialize the quotient line bundle $\ell=TC/\xi_C$ via $\theta_C$, so that the curvature form $\omega\in\Gamma(\wedge^{2}\xi_C^{*}\otimes \ell)$ is identified with $-d\theta_C|_{\xi_C\times\xi_C}$. Because
		\begin{align*}
			\theta_C\wedge d\theta_C&=\left(\sin\theta_1 d\theta_2+\cos\theta_1 d\theta_3\right)\wedge\left(\cos\theta_1 d\theta_1\wedge d\theta_2-\sin\theta_1 d\theta_1\wedge d\theta_3\right)\\
			&=d\theta_1\wedge d\theta_2\wedge d\theta_3,
		\end{align*}
		it follows that the curvature form has constant rank. Hence $\xi_C$ is a pre-contact distribution, and its characteristic foliation $\mathcal{F}$ is given by
		\begin{align*}
			T_{(p,q)}\mathcal{F}&=\ker\left(d\theta_C|_{\xi_C\times\xi_C}\right)_{(p,q)}\\
			&=\left\{v\in (\xi_C)_{(p,q)}:\iota_v(\theta_C\wedge d\theta_C)=0\right\}\\
			&=T_{p}S^{2}.
		\end{align*}
		A choice of complement $G$ to $T\mathcal{F}$ allows one to construct the contact thickening $(U,\ker\gamma_G)$ of the pre-contact manifold $(C,\xi_C)$, see \S\ref{sec:one}. Moreover, $(C,\xi_C)$ embeds as a regular coisotropic submanifold into the contact manifold $(U,\ker\gamma_G)$. It follows from Cor.~\ref{cor:hamilton} that $C$ is rigid, when deformed as a coisotropic submanifold of $(U,\ker\gamma_G)$.
	\end{ex}

	\color{black}
	
	\appendix
	
	\section{The Atiyah algebroid of a line bundle}
	
	We review some background material about the Atiyah algebroid of a line bundle and the corresponding der-complex. We also recall that pre-contact structures can be viewed as pre-symplectic forms on this algebroid, as shown in \cite{luca}.

	\begin{defi}
		Let $L\rightarrow M$ be a line bundle. An $\mathbb{R}$-linear map $\square:\Gamma(L)\rightarrow\Gamma(L)$ is called a \textbf{derivation} of $L$ if there exists some $X\in\mathfrak{X}(M)$ such that 
		\[
		\square(f\lambda)=X(f)\lambda+f\square(\lambda),\hspace{1cm}\forall \lambda\in\Gamma(L), f\in C^{\infty}(M).
		\]
		The vector field $X$ is unique; it is called the \textbf{symbol} of $\square$ and will be denoted by $\sigma(\square)$.
	\end{defi}
	
	The space of derivations of $L$ is the module of sections of a vector bundle $DL\rightarrow M$. The fiber $(DL)_{x}$ over a point $x\in M$ consists of the \textbf{derivations of $\mathbf{L}$ at $\mathbf{x}$}, i.e. $\mathbb{R}$-linear maps $\delta:\Gamma(L)\rightarrow L_{x}$ for which there exists a (necessarily unique) tangent vector $v\in T_{x}M$ such that $\delta(f\lambda)=v(f)\lambda_{x}+f(x)\delta(\lambda)$ for $f\in C^{\infty}(M)$ and $\lambda\in\Gamma(L)$. Here $v$ is also called the symbol of $\delta$, again denoted by $\sigma(\delta)$. The vector bundle $DL$ satisfies $\text{rk}(DL)=\dim(M)+1$. 
	
	\begin{defi}
		The vector bundle $DL\rightarrow M$ is a transitive Lie algebroid, called the \textbf{Atiyah algebroid} of $L$. Its Lie bracket is the commutator of derivations $[-,-]$, and the anchor map is given by the symbol $\sigma:DL\rightarrow TM$.
	\end{defi}
	
	The Lie algebroid $DL$ has a tautological representation $\nabla$ on $L$, given by
	\[
	\nabla_{\square}\lambda=\square\lambda,\hspace{1cm}\square\in\Gamma(DL),\lambda\in\Gamma(L).
	\] 
	Accordingly, we obtain the \textbf{der-complex} $\Omega^{\bullet}_{D}(L):=\Gamma(\wedge^{\bullet}(DL)^{*}\otimes L)$ of $L$, consisting of $L$-valued \textbf{Atiyah forms}. One has the usual Cartan calculus involving the following operations.
	\begin{enumerate}
		\item The differential $d_{D}:\Omega^{k}_{D}(L)\rightarrow\Omega^{k+1}_{D}(L)$ is defined by the usual Koszul formula
		\begin{align*}
			d_{D}\eta(\triangle_1,\ldots,\triangle_{k+1})&=\sum_{i=1}^{k+1}(-1)^{i+1}\triangle_{i}\big(\eta(\triangle_1,\ldots,\triangle_{i-1},\widehat{\triangle_{i}},\triangle_{i+1},\ldots,\triangle_{k+1})\big)\nonumber\\
			&\hspace{0.5cm}+\sum_{i<j}(-1)^{i+j}\eta\big([\triangle_{i},\triangle_{j}],\triangle_{1},\ldots,\widehat{\triangle_{i}},\ldots,\widehat{\triangle_{j}},\ldots,\triangle_{k+1}\big).
		\end{align*}  
		
		\item For any $\square\in\Gamma(DL)$, the contraction $\iota_{\square}:\Omega^{k}_{D}(L)\rightarrow\Omega^{k-1}_{D}(L)$ is given by
		\[
		(\iota_{\square}\eta)(\triangle_1,\ldots,\triangle_{k-1})=\eta(\square,\triangle_1,\ldots,\triangle_{k-1}).
		\]
		\item For any $\square\in\Gamma(DL)$, the Lie derivative $\mathcal{L}_{\square}:\Omega^{k}_{D}(L)\rightarrow\Omega^{k}_{D}(L)$ is given by
		\begin{equation}\label{eq:lie-der}
			(\mathcal{L}_{\square}\eta)(\triangle_1,\ldots,\triangle_k)=\square\big(\eta(\triangle_1,\ldots,\triangle_k)\big)-\sum_{i=1}^{k}\eta(\triangle_1,\ldots,[\square,\triangle_i],\ldots,\triangle_k).
		\end{equation}
	\end{enumerate}
	These satisfy the usual identities, for $\square,\triangle\in\Gamma(DL)$:
	\begin{equation}\label{eq:cartan}
		\begin{cases}
			d_{D}\iota_{\square}+\iota_{\square}d_{D}=\mathcal{L}_{\square},\\
			\mathcal{L}_{\square}\iota_{\triangle}-\iota_{\triangle}\mathcal{L}_{\square}=\iota_{[\square,\triangle]},\\
			\mathcal{L}_{\square}\mathcal{L}_{\triangle}-\mathcal{L}_{\triangle}\mathcal{L}_{\square}=\mathcal{L}_{[\square,\triangle]},\\ d_{D}\mathcal{L}_{\square}-\mathcal{L}_{\square}d_{D}=0.
		\end{cases}
	\end{equation}

	\begin{remark}\label{rem:contracting-homotopy}
		\begin{enumerate}[label=\roman*)]
			\item There is a canonical isomorphism $(DL)^{*}\otimes L\cong J^{1}L$, where $J^{1}L$ is the first jet bundle of $L$. The differential 
			$
			d_{D}:\Gamma(L)\rightarrow\Gamma((DL)^{*}\otimes L)
			$
			corresponds under this isomorphism to the first jet prolongation $j^{1}:\Gamma(L)\rightarrow\Gamma(J^{1}L)$.
			\item The der-complex $\big(\Omega^{\bullet}_{D}(L),d_{D}\big)$ is acyclic. Indeed, the contraction $\iota_{\mathbbm{1}}$ by the identity derivation $\mathbbm{1}\in\Gamma(DL)$ satisfies $d_{D}\iota_{\mathbbm{1}}+\iota_{\mathbbm{1}}d_{D}=\text{id}$.
		\end{enumerate}	
	\end{remark}
	
	\begin{defi}\label{def:pres-atiyah}
		Let $\varpi\in\Omega^{2}_{D}(L)$ be an Atiyah form such that $\iota_{\mathbbm{1}}\varpi$ is nowhere zero. Then $\varpi$ is called \textbf{pre-symplectic} (resp. symplectic) if $\varpi$ is $d_{D}$-closed and the bundle morphism $\varpi^{\flat}:DL\rightarrow(DL)^{*}\otimes L$ has constant rank (resp. is an isomorphism).
	\end{defi}
	
	A pre-symplectic Atiyah form $\varpi\in\Omega^{2}_{D}(L)$ defines a Lie subalgebroid
	$
	K:=\ker\varpi^{\flat}\subset DL.
	$
	
	\begin{prop}\cite{luca}\label{prop:relation-luca}
		Let $L\rightarrow M$ be a line bundle.
		\begin{enumerate}
			\item The relation
			\[
			\sigma^{*}\theta=\iota_{\mathbbm{1}}\varpi
			\]
			gives a one-to-one correspondence between $L$-valued pre-contact forms $\theta\in\Omega^{1}(M,L)$ and pre-symplectic Atiyah forms $\varpi\in\Omega^{2}_{D}(L)$.
			\item If $\theta$ and $\varpi$ correspond to each other, then the symbol map $\sigma:DL\rightarrow TM$ induces a Lie algebroid isomorphism
			\[
			K\subset DL\overset{\sim}{\longrightarrow}T\mathcal{F}\subset TM,
			\]
			where $T\mathcal{F}$ is the characteristic distribution of the pre-contact structure $\theta$.
		\end{enumerate}
	\end{prop}
	
	Assume we are given line bundles $L\rightarrow M$ and $L'\rightarrow M'$. A bundle map $(\varphi,\underline{\varphi}):L\rightarrow L'$ is called \textbf{regular} if $\varphi$ is fiberwise invertible. Such a map $(\varphi,\underline{\varphi})$ induces a morphism of Lie algebroids $D\varphi:DL\rightarrow DL'$ covering $\underline{\varphi}$, which is defined as follows. First note that there is a well-defined pullback map on sections $\varphi^{*}:\Gamma(L')\rightarrow\Gamma(L)$, given by
	\[
	(\varphi^{*}\lambda')_{x}=\left(\varphi|_{L_x}\right)^{-1}\lambda'_{\underline{\varphi}(x)},\hspace{1cm}\lambda'\in\Gamma(L'), x\in M.
	\] 
	Using this, one can define $D\varphi:DL\rightarrow DL'$ by setting
	\[
	\big(D\varphi(\delta)\big)(\lambda')=\varphi\big(\delta(\varphi^{*}\lambda')\big),\hspace{1cm}\delta\in(DL)_{x}, \lambda'\in\Gamma(L').
	\]
	Furthermore, one obtains a chain map $\varphi^{*}:\big(\Omega^{\bullet}_{D}(L'), d_D\big)\rightarrow\big(\Omega^{\bullet}_{D}(L),d_D\big)$ which extends the pullback on sections $\varphi^{*}:\Gamma(L')\rightarrow\Gamma(L)$. It is explicitly given by
	\[
	\big(\varphi^{*}\eta'\big)_{x}(\delta_1,\ldots,\delta_k)=\left(\varphi|_{L_x}\right)^{-1}\left(\eta'_{\underline{\varphi}(x)}\big(D\varphi(\delta_1),\ldots,D\varphi(\delta_k)\big)\right),
	\]
	for $\eta'\in\Omega^{k}_{D}(L')$ and $\delta_1,\ldots,\delta_k\in (DL)_x$ for some $x\in M$. 
	
	\begin{ex}\label{ex:inclusion}
		Consider the particular case in which $\underline{\varphi}:N\hookrightarrow M$ is the inclusion of a submanifold, and $\varphi:L_{N}\hookrightarrow L$ is the inclusion of the restricted vector bundle $L_{N}:=L|_{N}$. Then the map $D\varphi:DL_{N}\hookrightarrow DL$ is an embedding whose image consists of those $\delta\in DL$ for which $\sigma(\delta)\in TN$. This allows us to view $DL_{N}$ as a Lie subalgebroid of $DL$. Moreover, we have a pullback map $\varphi^{*}:\Omega^{\bullet}_{D}(L)\rightarrow\Omega^{\bullet}_{D}(L_{N})$.
	\end{ex}
	
	At last, we recall the interpretation of derivations as infinitesimal bundle automorphisms. Given a line bundle $L\rightarrow M$, let $\square_t\in\Gamma(DL)$ be a smooth family of derivations. It generates a smooth family of local line bundle automorphisms $\varphi_t$, uniquely determined by 
	\begin{equation}\label{eq:flow}
		\frac{d}{dt}\varphi_t^{*}\lambda=\varphi_t^{*}\big(\square_t\lambda\big),\hspace{1cm}\forall\lambda\in\Gamma(L),
	\end{equation}
	with initial condition $\varphi_0=\text{Id}$. The automorphisms $\varphi_t$ cover the local flow $\underline{\varphi}_t$ of the time-dependent vector field $\sigma(\square_t)$. The formula \eqref{eq:flow} extends to time-dependent Atiyah forms via
	\[
	\frac{d}{dt}\varphi_t^{*}\eta_t=\varphi_t^{*}\left(\mathcal{L}_{\square_t}\eta_t+\frac{d}{dt}\eta_t\right),\hspace{0.5cm}\eta_t\in\Omega^{k}_{D}(L).
	\]

	\color{black}

	\bibliography{CoisotropicContact}
	\bibliographystyle{plain}

\end{document}